\documentclass[12pt,a4paper,twoside,reqno]{amsart} 
\usepackage{amsfonts, amsthm, amsmath, amssymb}
\usepackage{hyperref}
\hypersetup{colorlinks=false}

\usepackage[margin=2.9cm]{geometry}

\usepackage{helvet}

\usepackage{graphicx}
\usepackage{amsmath}

\usepackage{amsthm}
\newtheorem{theorem}{Theorem}
\newtheorem{lemma}{Lemma}[section]
\newtheorem*{remarks}{Remarks}

\begin{document}

\author{ Sumit Kumar, Kummari Mallesham and Saurabh Kumar Singh}
\title{Non-linear additive twists of $GL(3) \times GL(2)$ and  $GL(3)$ Maass forms}

\address{Sumit Kumar \newline  Stat-Math Unit, Indian Statistical Institute, 203 B.T. Road, Kolkata 700108, India; email: sumitve95@gmail.com}

\address{ Kummari Mallesham \newline Stat-Math Unit, Indian Statistical Institute, 203 B.T. Road, Kolkata 700108, India;  email:iitm.mallesham@gmail.com
}

\address{ Saurabh Kumar Singh \newline { Department of Mathematics and Statistics, Indian Institute of Technology, Kanpur  208016, India; \newline  Email: saurabs@iitk.ac.in
} }

\date{\today}

.
\maketitle

\begin{abstract}
	Let $\lambda_{\pi}(m,n)$ be  the Fourier coefficients of a Hecke-Maass cusp form $\pi$ for $SL(3,\mathbb{Z})$ and $\lambda_{f}(n)$ be the Fourier coefficients of Hecke-eigen form $f$ for $SL(2,\mathbb{Z})$.  The aim of this article is to get a non-trivial bound on the sum which is non-linear additive twist of the coefficients $\lambda_{\pi}(m,n)$ and $\lambda_{f}(n)$.  More precisely, for any $0 < \beta < 1$ and  $\epsilon>0$, we have 
$$\sum_{n=1}^{\infty} \lambda_{\pi}(r,n) \, e\left(\alpha n^{\beta}\right) V\left(\frac{n}{X}\right) \ll_{\pi,\epsilon} \alpha \sqrt{\beta}r^{\frac{7}{6}}X^{\frac{3}{4}+\frac{9\beta}{28}+ \epsilon}.$$
and 
$$\sum_{n=1}^{\infty} \lambda_{\pi}(r,n) \, \lambda_{f}(n) \, e\left(\alpha n^{\beta}\right) V\left(\frac{n}{X}\right) \ll_{\pi, f,\epsilon}   (\alpha \beta)^{\frac{3}{2}} rX^{\frac{3}{4}+\frac{29\beta}{44}+\epsilon},$$
where $V(x)$ is a smooth function supported in $[1,2]$ and satisfying $V^{(j)}(x) \ll_{j} 1$.
\end{abstract}

\section{\bf Introduction}

When $a = {a(n)}_{n \geq 1}$ is a sequence of complex numbers arising in an arithmetical context, a well-known problem is to obtain upper bounds for the sum 

\begin{equation} \label{general sequence}
\sum_{n \leq X} a(n) e \left( \alpha n^\beta \right)
\end{equation}

for integers $X>1$ and $\alpha, \beta$ given real numbers.  Here, as usual, $e(z)$ denotes $e^{2 \pi i z}$ for any complex number $z$. The sequence $a = { a(n) e \left( \alpha n^\beta \right)}_{n \geq 1}$  is, somewhat loosely, called a non-linear additive twist of the sequence a or simply an additive twist of $a$ if the case $\beta = 1$ is alone under consideration.

Fourier coefficients of automorphic forms provide a large supply of arithmetically interesting
sequences $ {a(n)}_{n \geq 1}$  and, indeed, in the $GL(2)$ setting, when the $ a(n)$ are the Fourier coefficients of a holomorphic modular form or a Maass form on the upper half plane, the study of the sum \eqref{general sequence} and its smoothed versions, for various $\alpha, \beta$ has a long history. We refer to H. Iwaniec \cite{iwaniec123} and to the more recent papers D. Godber \cite{godber} for a more detailed account of the results in this case.

The present article is concerned with smoothed dyadic versions of the sum \eqref{general sequence} when the $a(n)$ arise from the Fourier coefficients of certain $GL(3)$ or $GL(3) \times GL(2) $ automorphic forms. Thus let $ \lambda_\pi (m, n)$ be the Fourier-Whittaker coefficients of a Maass form $\pi$ for $SL(3, Z)$. Also, let $w$ be a function defined by
\begin{equation} \label{bump function}
w : \mathbb{R} \rightarrow \mathbb{C} \textrm{ is infinitely differentiable on } \mathbb{R} \textrm{ with support in } [1, 2].
\end{equation}
Then for all real $\alpha$, real $X > 1$ and integers $m$ in $[1, X]$ we have the following bound from \cite[Theorem 2]{Ren}

\begin{equation} \label{renbound}
\sum_{n=1}^{\infty} \lambda_{\pi}(m,n) \, e\left(\alpha n^{\beta}\right) V\left(\frac{n}{X}\right) \ll_{\pi, \alpha, \epsilon}  m^{\frac{5}{14} + \epsilon}  \left( 1+ ( \alpha X)^{\frac{3 \beta}{2}}  \right) \log X.
\end{equation}  for any real $ \beta > 0$, $\beta \neq 1/3$
and  $\epsilon> 0$. Since $w$ vanishes outside $[1, 2],$  the left hand side
of \eqref{renbound} is, in fact, a finite sum with $n$ ranging over the integers in the interval $[X, 2X]$, that is, a smoothed dyadic version of the sum \eqref{general sequence} with $a(n) = \lambda_\pi(m, n)$. Now, an application of the Cauchy-Schwarz inequality taken together with a well-known mean square bound for the $\lambda_\pi(m, n)$, reviewed in Section below, gives

\begin{equation} \label{trivial bound}
\sum_{n=1}^{\infty} \lambda_{\pi}(m,n) \, e\left(\alpha n^{\beta}\right) V\left(\frac{n}{X}\right) \ll_{\pi, \alpha, \epsilon}  X^{1/2} \left(\sum_{n\leq 2 X} \left| \lambda_{\pi}(m,n)\right|^2 \right)^{1/2} \ll_{\pi, \alpha, \epsilon} X^{1+ \epsilon} 
\end{equation}
for all $\epsilon > 0$ and any real $ X \geq 1$, real  $\alpha, \beta$ and integer $m \geq 1$. This bound may be thought of as the trivial estimate for the left hand side of \eqref{renbound}. Plainly, \eqref{renbound} gains over \eqref{trivial bound} in its dependence
on $X$ only if $0 < \beta \leq 2/3$. Our first result gains over \eqref{trivial bound}  in this aspect for  $\beta$ in a wider range :

\begin{theorem} \label{Thm}
Let $\lambda_\pi(r, n)$ be the Fourier-Whittaker coefficients of a Maass form $\pi$ for
$SL(3, \mathbb{Z})$. Then for any real $\alpha$, $0 < \beta <1$ and integer $ r \geq 1$ we have 
\begin{equation} \label{mainthe}
S_r(X) := \sum_{n=1}^{\infty} \lambda_{\pi}(r,n) \, e\left(\alpha n^{\beta}\right) V\left(\frac{n}{X}\right) \ll_{\pi, \epsilon}  \alpha \sqrt{\beta}r^{7/6}X^{3/4+9\beta/28}.
\end{equation}
for all $X\geq 1$, where $w(x)$ is as in equation \eqref{bump function}. 
\end{theorem}

 In its dependence on X the above bound improves on the trivial estimate \eqref{trivial bound}  when $ 0 <  \frac{3 }{4}+\frac{9 \beta}{24} <1$, i.e., $ \beta < \frac{7}{9}$. Also, bound in \eqref{mainthe} is a stronger bound than \eqref{renbound} when  $ \frac{3 }{4}+\frac{3 \beta}{10} < \frac{3 \beta}{2} $ and $\beta <1$, i.e., $\frac{7 }{11} < \beta <1 $. 

Upper bounds for the sum on the left hand side of \eqref{mainthe} when $\beta = \frac{2}{3}$
are especially interesting. In this case, \eqref{renbound} gives essentially the trivial upper bound of \eqref{trivial bound}, while the exponent of $X$ yielded
by our bound \eqref{mainthe} is $ 27 /28 + \epsilon$. We remark that if this exponent could be lowered to $\theta + \epsilon$   for some $\theta < 5/6$, it would be possible to show that the standard L-function $L(s, \pi)$ attached to $\pi$ has infinitely many zeros on the critical line $ \Re s = 1/2$ by adapting methods that lead to Hardy's theorem on the infinitude of zeros of the Riemann zeta function on this line (see \cite{rajkumar}).

Our method for proving Theorem \ref{Thm} is quite different from that in \cite{Ren}, which is based on the $GL(3)$-Voronoi summation formula. We use the method developed by R. Munshi in \cite{munshi1}, \cite{munshi2} and  \cite{munshi3}. Briefly speaking, this method involves an application of circle method to separate the oscillatory terms in the sum on the left hand side of \eqref{mainthe}  followed by the application of various
summation formulae. Our method is flexible enough to allow us to prove a result analogous to
Theorem \ref{mainthe} for the Rankin-Selberg convolution $\pi  \times f$ of $ GL(3)$ and $GL(2)$ Maass forms. Thus, we also have the following theorem:

\begin{theorem} \label{gl3thm21}
Let $\lambda_f (n)$ be the Fourier coefficients of a holomorphic/  Maass form $f$ for $SL(2, \mathbb{Z}) $ and let $\lambda_{\pi}(r,n) $ be as in the statement of Theorem \ref{mainthe}.  Then for any $\beta$ such that $1/3 < \beta <11/29$, $\alpha > 1/\beta$  and integer $ r \geq 1$ we have 
\begin{equation} \label{gl3gl2}
\sum_{n=1}^{\infty} \lambda_{\pi}(r,n) \lambda_f (n)\, e\left(\alpha n^{\beta}\right) V\left(\frac{n}{X}\right) \ll_{\pi, \alpha,\epsilon}    (\alpha \beta)^{\frac{3}{2}} rX^{\frac{3}{4}+\frac{29\beta}{44}+\epsilon}. 
\end{equation}
for all $X\geq 1$, where $w(x)$ is as in equation \eqref{bump function}. 
\end{theorem}

\begin{remarks}
\begin{enumerate}
	\item We compare the bound in \eqref{mainthe} with that  in \eqref{renbound}. Our bound is better in the range $\beta > 7/11$. Infact, it gives power saving bound when $\beta <7/9$. In our treatment of the sum $S(X)$, we separated $GL(3)$ Fourier coefficients $\lambda_{\pi}(m,n)$ from $e\left(\alpha n^{\beta}\right)$ using the circle method which gives us  more flexibility while applying the summation formulas and treating the integrals.  
	\item   We have only focused on the range of the parameter $\beta$ for which we get the power saving bound. Infact, the case $\beta =2/3$ is of particular interest as it is related to Hardy's type theorem for $GL(3)$ $L$-function. In this case, we have $S_r(X) \ll_{\epsilon} X^{27/28 + \epsilon}$, which is a step towards  proving Hardy's type theorem .
	 
	 \item This is the first instance where non-trivial estimates of non-linear twists of $GL(3) \times GL(2)$ have been achieved. The estimate given in equation \eqref{gl3gl2} gives power saving bound when $\beta < 11/29$.

	

\end{enumerate} 
\end{remarks}
\subsection{Sketch of the proof of Theorem \ref{Thm}}
The proof of Theorem \ref{Thm} starts in  Section \ref{aplicircle}.  Now we give a sketch of the proof  briefly. For simplicity, we assume that $r=1$ and $n \sim X$. Thus our object is to get cancellations in the following sum 
$$S(X) = \sum_{n \sim X} \lambda_{\pi}(1,n) e\left(\alpha n^{\beta}\right).$$
As a first step, following the methods of Munshi \cite{munshi2}, we separate the oscillations of $\lambda_{\pi}(1,n)$ and $e(\alpha n^{\beta})$ by using the circle method as follows
$$ S(X) = \frac{1}{K} \int V \left(\frac{v}{K} \right) \mathop{\sum \sum}_{n,m \sim X} \lambda_{\pi}(1,n) n^{iv} e\left(\alpha m^{\beta}\right) m^{-iv} \delta(n-m) \mathrm{d}v$$
where $X^{\epsilon} < K < X^{\beta}$ and $\int V =1$. Writing the Fourier expansion for $\delta(n-m)$ we see that $S(X)$ is, roughly, given by
\begin{align*}
\frac{1}{Q^2 K} \int_{K}^{2K} & \sum_{q \sim Q} \,  \sideset{}{^\star}{\sum}_{a \, \rm mod \, q} \, \sum_{n \sim X} \lambda_{\pi}(1,n) n^{iv} e\left(\frac{an}{q}\right) \\
& \times \sum_{m \sim X} m^{-iv} e\left(\alpha m^{\beta}\right) e\left(\frac{-am}{q}\right) \, \mathrm{d}v
\end{align*}
where $Q=\sqrt{X/K}$. Trivially estimating the above sum, we get
$$S(X) \ll X^{2+\epsilon}.$$
So, to get cancellation in the sum $S(X)$ we need to save $X$ (and little more) in a sum of the form
$$\int_{K}^{2K} \sum_{q \sim Q} \,  \sideset{}{^\star}{\sum}_{a \, \rm mod \, q} \, \sum_{n \sim X} \lambda_{\pi}(1,n) n^{iv} e\left(\frac{an}{q}\right) \, \sum_{m \sim X} m^{-iv} e\left(\alpha m^{\beta}\right) e\left(\frac{-am}{q}\right).$$

The second step is to apply summation formulas on the $m,n$ sums. An application of the Poisson summation formula on the $m$-sum gives a saving $X/\sqrt{Q X^{\beta}}$ and an application of the $GL(3)$ Voronoi summation formula gives a saving $X/\sqrt{(QK)^3}$. Moreover, we get $\sqrt{Q}$ saving in the $a$-sum and $\sqrt{K}$ saving in the $v$-integral. Thus, at end of the summation formulas our total saving is  
$$\frac{X}{\sqrt{Q X^{\beta}}} \frac{X}{\sqrt{(QK)^3}} \sqrt{Q} \sqrt{K} = X \frac{X^{\frac{1}{4}-\frac{\beta}{2}}}{K^{\frac{1}{4}}}.$$
Therefore, we need to save $K^{1/4} X^{\frac{\beta}{2}-\frac{1}{4}}$ and little more in the following transformed sum (after summation formulas)
$$\sum_{n \sim \frac{Q^3 K^3}{X}} \lambda_{\pi}(1,n) \sum_{q \sim Q} \sum_{\substack{(m,q)=1 \\ m \sim \frac{Q X^{\beta}}{X}}} S(\bar{m},n;q) I(n,m,q)$$
where $I(n,m,q)$ is an integral transform.

The next step involves applying Cauchy-Schwarz inequality in the $n$-sum to get rid of the Fourier coefficients $\lambda_{\pi}(1,n)$, here we use Ramnujan bound on the average for $\lambda_{\pi}(1,n)$.  So we have to save $K^{1/2} X^{\beta -1/2}$ in the following sum
$$\sum_{n \sim \frac{Q^3 K^3}{X}}  \vert \sum_{q \sim Q} \sum_{\substack{(m,q)=1 \\ m \sim \frac{Q X^{\beta}}{X}}} S(\bar{m},n;q) I(n,m,q) \vert^2.$$
Now we open the absolute square and apply the Poisson summation formula over the $n$ sum. The resulting zero frequency contribution is satisfactory if 
$$\frac{Q^2 X^{\beta}}{X} > K^{1/2} X^{\beta -1/2}$$
or equivalently $K < X^{1/3}$. So,  we get following  saving in the non-zero frequencies  
$$\frac{Q^3 K^3}{X Q K^{2/3}} = K^{4/3}.$$
This is sufficient if $K^{4/3} > K^{1/2} X^{\beta -1/2}$. Thus we get a condition 
$$X^{6\beta/5 - {3}/{5}} < K < X^{\frac{1}{3}}.$$ 
Therefore, we get a power saving bound on the sum $S(X)$ when $\beta < 7/9$. Now we choose $K$ by equating the diagonal and off-diagonal savings
$$\frac{Q^2 X^{\beta}}{X} = \frac{Q^3 K^3}{XQ K^{2/3}} \implies K = X^{\frac{3\beta}{7}}$$
and the total saving is 
$$X \frac{X^{\frac{1}{4}-\frac{\beta}{2}}}{K^{\frac{1}{4}}} K^{2/3} = X^{\frac{5}{4}-\frac{9 \beta}{28}}.$$
Therefore, we get
$$S(X) \ll \frac{X^{2+\epsilon}}{X^{\frac{5}{4}-\frac{3 \beta}{10}}}= X^{\frac{3}{4}+\frac{ 9\beta}{28}+\epsilon}.$$

\subsection{Sketch of the proof of Theorem \ref{gl3thm21}}
	We now give the sketch of the proof of Theorem \ref{gl3thm21}. Like Theorem \ref{Thm}, we need to get cancellations in the following sum 
	$$R(X) = \sum_{n \sim X} \, \lambda_{\pi}(1,n) \, \lambda_{f}(n) \, e\left(\alpha n^{\beta}\right). $$
	On separating the oscillations of $\lambda_\pi(1,n)$ and $\lambda_{f}(n) \, e\left(\alpha n^{\beta}\right)$ using the the circle method we arrived at the following expression 
	\begin{align*}
	R(X) = \frac{1}{K} \int V \left(\frac{v}{K} \right) \mathop{\sum \sum}_{n,m =1}^\infty \lambda_{\pi}(1,n) n^{iv} \, \lambda_{f}(m) \,& e\left(\alpha m^{\beta}\right) m^{-iv} U \left( \frac{n}{X}\right)  \\ 
	&   \times V\left( \frac{m}{X}\right)\delta(n-m) \, \mathrm{d}v
	\end{align*}
	where $X^{\epsilon} < K < X^{\beta}$ and $\int_{\mathbb{R}} V (u) du=1$. On substituting  the Fourier expansion for $\delta(n-m)$ from equation \eqref{deltasymbol},  we obtain
	\begin{align*}
	R(X)=\frac{1}{Q K} & \int_{\mathbb{R}} g(q, x)\int V \left(\frac{v}{K} \right)   \sum_{q \leq Q}  \frac{1}{q}\,  \sideset{}{^\star}{\sum}_{a \, \rm mod \, q} \, \sum_{n =1}^\infty  \lambda_{\pi}(1,n) n^{iv} e\left(\frac{an}{q}\right) e\left(\frac{xn}{qQ}\right)  U \left( \frac{n}{X}\right)  \\
	& \times \sum_{m =1}^\infty \, \lambda_{f}(m) \, m^{-iv} e\left(\alpha m^{\beta}\right) e\left(\frac{-am}{q}\right) \,  e\left(\frac{-mx}{qQ}\right)  V \left( \frac{m}{X}\right) \mathrm{d}v \,  {d}x, 
	\end{align*}
	where $Q=\sqrt{X/K}$. Trivially estimating the above sum, we get
	$$R(X) \ll X^{2+\epsilon}.$$
	So, to get cancellation in the sum $R(X)$ we need to save $X$ (and little more) in a sum of the form
	$$\int_{K}^{2K} \sum_{q \sim Q} \,  \sideset{}{^\star}{\sum}_{a \, \rm mod \, q} \, \sum_{n \sim X} \lambda_{\pi}(1,n) n^{iv} e\left(\frac{an}{q}\right) \, \sum_{m \sim X} \, \lambda_{f}(m) \, m^{-iv} e\left(\alpha m^{\beta}\right) e\left(\frac{-am}{q}\right) \, dv.$$
	Next we apply Voronoi summation formulas on the  $m$ and $n$ sums. On applying $GL(3)$ Voronoi summation formula on the $n$-sum, the dual length becomes $n^{*} \sim Q^3 K^3 /X$ and we save $X/\sqrt{Q^3 K^3}$.
	
	An application of $GL(2)$ Voronoi on the $m$-sum  converts the $m$-sum into a dual sum of the length $m^{*} \sim Q^2 X^{2 \beta}/X$ and this gives a saving of size $X/Q X^{\beta} $. 
	
	Like before, we save $\sqrt{Q}$ in the $a$-sum and $\sqrt{K}$ in the $v$-integral. So far we have the following saving
	$$\frac{X}{\sqrt{Q^3 K^3}} \frac{X}{Q X^{\beta}} \sqrt{Q} \sqrt{K} = \frac{X}{X^{\beta}}.$$
	Therefore, we need to save $X^{\beta}$ and little more in the following transformed sum 
	$$ \sum_{q \sim Q} \sum_{n^{*} \sim \frac{Q^3 K^3}{X}} \, \lambda_{\pi}(1,n) \sum_{m^{*} \sim \frac{Q^2 X^{2 \beta}}{X}} \, \lambda_{f}(m) \, \mathfrak{C} \, \mathfrak{I}$$
	where $\mathfrak{I}$ is an integral transform which oscillates like $n^{i K}$ with respect to $n$, and the character sum is given by 
	$$\mathfrak{C} = \sideset{}{^{*}}{\sum}_{a \, \rm mod \, q} S\left(\bar{a},n;q\right) \, e\left(\frac{\bar{a} m}{q}\right) \rightsquigarrow q e\left(-\frac{\bar{m} n}{q}\right).$$
	In the next step we apply the Cauchy inequality in the $n$-sum to get rid of the coefficients $\lambda_{\pi}(1,n)$ and we arrive at the following expression 
	$$\sum_{n^{*} \sim \frac{Q^3 K^3}{X}} \Big \vert \sum_{q \sim Q} \sum_{m^{*} \sim \frac{Q^2 X^{2 \beta}}{X}} \lambda_{f}(m) \, e\left( - \frac{\bar{m} n}{q}\right) \, \mathfrak{I}\Big \vert^{2},$$
	where we seek to save $X^{2 \beta}$ plus little more. Opening the absolutely value square we apply the Poisson summation formula on the $n$-sum. In the zero frequency we save $Q^3 X^{2 \beta}/X$ which is satisfactory if 
	$$\frac{Q^3 X^{2 \beta}}{X} > X^{2 \beta}$$
	or equivalently $K < X^{1/3}$. In the non-zero frequency we save 
	$$\frac{Q^3 K^3}{X K^{2/3}}$$ 
	which is sufficient if 
	$$\frac{Q^3 K^3}{X K^{2/3}} > X^{2 \beta}$$
	i.e., $K > X^{12 \beta/5 - {3}/{5}}$. Thus we get the following  restriction on the choice  of $K$
	$$X^{12 \beta/5 - {3}/{5}} < K < X^{1/3}.$$
	Hence, we have a room to choose $K$ optimally. 
We end the introduction by defining some notations.
\paragraph*{\bf{Notations}} Throughout the paper, $e(x)$ means $e^{2\pi i x}$ and negligibly small means $O(X^{-A})$ for any $A >0$. In particular, we will take $A=2020$. The notation $\alpha \ll A$ will mean that for any $\epsilon >0$, there is a constant $c$ such that $|\alpha| \leq cAX^{\epsilon}$. By $\alpha \asymp A$, we mean that $X^{-\epsilon}A \leq \alpha \leq X^{\epsilon}A$, also $\alpha \sim A$ means $A\leq \alpha < 2A$.

 \section{ \bf Preliminaries}
 \subsection{The Delta method} \label{circlemethod}
 Let $\delta: \mathbb{Z} \to \{0,1\}$ be defined by
 \[
 \delta(n)= \begin{cases}
 1 \quad \text{if} \,\  n=0; \\
 0 \quad  \,  $\textrm{otherwise}$.
 \end{cases}
 \]
 The above delta symbol can be used to separate the oscillations involved in a sum. Further, we seek a Fourier expansion of $\delta(n)$. We mention here an expansion for $\delta(n)$ which is due to Duke, Friedlander and Iwaniec. Let $L\geq 1$ be a large number. For $n \in [-2L,2L]$, we have
 
 \begin{align*} 
 \delta(n)= \frac{1}{Q} \sum_{1 \leq q \leq Q} \frac{1}{q} \, \sideset{}{^\star}{\sum}_{a \, \rm mod \, q} \, e \left(\frac{na}{q}\right) \int_{\mathbb{R}} g(q,x) \,  e\left(\frac{nx}{qQ}\right) \, \mathrm{d}x,
 \end{align*}
 where  $Q=2L^{1/2}$. The $\star$ on the sum indicates that the sum over $a$ is restricted by the condition $(a,q)=1$. The function $g$ is the only part in the above formula which is not explicitly given. Nevertheless, we only need the following two properties of $g$ in our analysis.  
 
 \begin{align} \label{g properties}
 &g(q,x)=1+h(q,x), \quad \text{with} \ \ \  h(q,x)=O \left(\frac{1}{qQ} \left(\frac{q}{Q}+|x|\right)^{B}\right), \\
 & g(q,x) \ll |x|^{-B} \notag
 \end{align}
 for any $B>1$. Using the second property of $g(q,x)$ we observe that the effective range of the integration in \eqref{deltasymbol}  is $[-L^{\epsilon},L^{\epsilon}]$. We record the above observations in the following lemma.
 \begin{lemma}\label{deltasymbol}
 	Let $\delta$ be as above and $g$ be a function satisfying \eqref{g properties}. Let $L\geq 1$ be a large parameter. Then, for $n \in [-2L,2L]$, we have
 	\begin{equation*} 
 	\delta(n)= \frac{1}{Q} \sum_{1 \leq q \leq Q} \frac{1}{q} \, \sideset{}{^\star}{\sum}_{a \, \rm mod \, q} \, e \left(\frac{na}{q}\right) \int_{\mathbb{R}}W(x) g(q,x) \,  e\left(\frac{nx}{qQ}\right) \, \mathrm{d}x+O(L^{-2020}),
 	\end{equation*}
 	where $Q=2L^{1/2}$ and  $W(x)$ is a smooth bump function supported in $[-2L^{\epsilon},2L^{\epsilon}]$, with $W(x)=1$ for $x \in [-L^{\epsilon},L^{\epsilon}] $ and $W^{(j)}\ll_j 1$.
 \end{lemma}
 \begin{proof}
 	For the proof, we refer to chapter 20 of the book \cite{iwaniec}.
 \end{proof}

\subsection{ Holomorphic forms on $GL(2)$}
Let $f$ be a holomorphic Hecke eigenform  of weight $k$ for the full modular group $SL(2,\mathbb{Z})$. The Fourier expansion of $f$ at $\infty$ is 
$$f(z) = \sum_{n=1}^{\infty} \, \lambda_{f}(n) \, n^{(k-1)/2} \, e(nz),$$
for $z \in \mathbb{H}$. We have a well-known  Deligne's bound for the Fourier coefficients which says that  
\begin{align}\label{gl2 ramanujan}
\vert \lambda_{f}(n) \vert \leq d(n),
\end{align}
for  $n\geq 1$, where $d(n)$ is the divisor function. We now state the Voronoi summation formula for $f$ in the following lemma. 

\begin{lemma} \label{gl2 voronoi}
	Let $\lambda_{f}(n)$ be as above and $g$ be a smooth, compactly supported function on $(0, \infty)$. Let $a$, $q \in \mathbb{Z}$ with $(a,q)=1$. Then we have
	$$\sum_{n=1}^{\infty} \lambda_{f}(n) \, e\left(\frac{an}{q}\right)g(n) = \frac{2\pi i^k}{q} \sum_{n=1}^{\infty} \lambda_{f}(n) \, e\left(-\frac{d n}{q}\right)\, h(n),$$
	where $ad \equiv 1 (\rm mod \, q)$ and 
	$$h(y) = \int_{0}^{\infty} g(x) \, J_{k-1} \left(\frac{4 \pi \sqrt{xy}}{q}\right) \, dx.$$
\end{lemma}
\begin{proof}
	Proof can be found in the book of Iwaniec-Kowalski \cite{iwaniec}. 
\end{proof}
 \subsection{ Maass forms for $GL(3)$}
 Let $\pi$ be a Maass form of type $(\nu_{1}, \nu_{2})$ for $SL(3, \mathbb{Z})$. By the work of Jacquet, Piatetski-Shapiro and Shalika, we have the Fourier Whittaker expansion of $\pi(z)$:
 \begin{equation} \label{four-whi-exp}
 \pi(z) = \sum_{\gamma \in U_{2}\left(\mathbb{Z}\right) \backslash  SL(2,\mathbb{Z})} \sum_{m_{1}=1}^{\infty} \sum_{m_{2} \neq 0} \frac{\lambda_{\pi}(m_{1},m_{2})}{m_{1} |m_{2}|} \, W_{J}\left(M \left( {\begin{array}{cc}
 	\gamma &  \\
 	 & 1 \\
 	\end{array} } \right)
 z, \nu, \psi_{1,1}\right),
 \end{equation}
 where $U_{2}(\mathbb{Z})$ is the group of upper triangular matrices with integer entries and ones on the diagonal,
 $W_{J}\left(z,\nu,\psi_{1,1}\right)$ is the Jacquet-Whittaker function, and $M=\textbf{diag} \left(m_{1}|m_{2}|,m_{1},1\right)$ (cf. Goldfeld \cite{gold}).
  
 Let $\psi(x)$ be a compactly supported smooth function on  $ (0, \infty )$ and $$\tilde{\psi}(s) = \int_{0}^{\infty} \psi(x) x^{s-1} \mathrm{d}x$$ be its Mellin transform. Set 
 $$\alpha_{1} = - \nu_{1} - 2 \nu_{2}+1, \alpha_{2} = - \nu_{1}+ \nu_{2}, \alpha_{3} = 2 \nu_{1}+ \nu_{2}-1.$$
 For $k=0$ and $1$, we define
 \begin{equation}
 \Psi_{k}(x) := \int_{(\sigma)} \left(\pi^{3} x\right)^{-s} \prod_{i=1}^{3} \frac{\Gamma\left(\frac{1+s+2k+\alpha_{i}}{2}\right)}{\Gamma\left(\frac{-s-\alpha_{i}}{2}\right)} \tilde{\psi}(-s-k) \,  \mathrm{d}s
 \end{equation}
 with $\sigma > -1 + \max \{-\Re(\alpha_{1}), -\Re(\alpha_{2}), \Re(\alpha_{3})\}$, and
 \begin{equation} \label{psicombina}
 \Psi_{0,1}^{\pm}(x)= \Psi_{0}(x)\pm \frac{\pi^{-3} q^{3}m}{n_{1}^2 n_{2} i} \Psi_{1}(x).
 \end{equation}
 With the aid of the above terminology we now state $GL(3)$-Voronoi summation formula in the following lemma. 
 \begin{lemma} \label{gl3voronoi}
 	Let $\psi(x)$ and  $\lambda_{\pi}(n,r)$ be as above. Let $a,\bar{a}, q \in \mathbb{Z}$ with $q \neq 0, (a,q)=1,$ and  $a\bar{a} \equiv 1(\mathrm{mod} \ q)$. Then we have
 	\begin{align*} \label{GL3-Voro}
 	\sum_{n=1}^{\infty} \lambda_{\pi}(r.n) e\left(\frac{an}{q}\right) \psi(n) 
 	=q  \sum_{\pm} \sum_{n_{1}|qr} \sum_{n_{2}=1}^{\infty}  \frac{\lambda_{\pi}(n_2,n_1)}{n_{1} n_{2}} S\left(r \bar{a}, \pm n_{2}; qr/n_{1}\right) \psi_{0,1}^{\pm} \left(\frac{n_{1}^2 n_{2}}{q^3 r}\right),
 	\end{align*} 
 	where  $S(a,b;q)$ is the  Kloosterman sum which is defined as follows:
 	$$S(a,b;q) = \sideset{}{^\star}{\sum}_{x \,\rm mod \, q} e\left(\frac{ax+b\bar{x}}{q}\right).$$
 \end{lemma}
 \begin{proof}
 	See \cite{miller-schmid} for the proof. 
 \end{proof}
  We need asymptotic behaviour of the function $\Psi_{0,1}^{\pm}(x)$ in our analysis, but $\Psi_{0,1}^{\pm}(x)$ is itself the combination of the functions $\Psi_{0}(x)$ and $\Psi_{1}(x)$ as given in \eqref{psicombina}. So we need the asymptotic behaviour of these individual functions, but we know  from \cite[ page.no: 307]{X.Li} that the asymptotic behaviour of $x^{-1} \Psi_{1}(x)$ is similar to that of $\Psi_{0}(x)$. Hence we only need to know  the behaviour of $\Psi_{0}(x)$ which is given in the following lemma.  
 
\begin{lemma} \label{GL3oscilation}
 	Let $\psi(x)$ be as above, with support in the interval $[X,2X]$. Then for any fixed integer $J \geq 1$ and $xX \gg 1$, we have
 	\begin{equation*}
 	\psi_{0}(x)=2\pi^4i x \int_{0}^{\infty} \psi(y) \sum_{j=1}^{J} \frac{c_{j} e\left(3 (xy)^{1/3} \right) + d_{j} e\left(-3 (xy)^{1/3} \right)}{\left( \pi^3xy\right)^{j/3}} \, \mathrm{d} y + O \left((xX)^{\frac{-J+2}{3}}\right),
 	\end{equation*}
 	where $c_{j}$ and $d_{j}$ are  absolute constants depending on $\alpha_{i}$, for $i=1, 2, 3$.  
 \end{lemma}
 \begin{proof}
 	See  \cite{X.Li}.
 \end{proof}

\begin{remarks}
\begin{enumerate}
	\item In the case  $xY \ll 1$, we have $$\Psi_{0}(x) \ll \int_{0}^{\infty} |\psi^{\prime}(x)| \mathrm{d}x + 1,$$ as mentioned in the remark of X. Li \cite[ page no: 307]{X.Li}. It can be seen easily that we save more in this case, in our analysis.                           
	
	\item In our analysis, we  will work  with the leading term, i.e, $j=1$ only in Lemma \ref{GL3oscilation} as the contributions of other terms gives us even better estimates.  
\end{enumerate}
\end{remarks}
 
 We end this subsection by mentioning the Ramanujan bound on average for the Fourier coefficients $\lambda_{\pi}(n_{1},n_{2})$ in  the following lemma.
 \begin{lemma} \label{ramanujan}
 	We have 
 \begin{equation} 
 \mathop{\sum \sum}_{n_{1}^2 n_{2} \leq x} \vert \lambda_{\pi}(n_{1},n_{2})\vert^{2} \; \ll \; x^{1+\epsilon}.
 \end{equation}
 \end{lemma}
\begin{proof}
	see \cite{gold}.
\end{proof}

\subsection{Bessel function} Let $k \geq 2$ be a fixed integer. Let $J_{k-1}(x)$ be the Bessel function of the first kind of order $k-1$. We have the following lemma. 
\begin{lemma} \label{bessel function decompo}
Let $J_{k-1}(x)$ be the Bessel function of first kind. Then for $x$ large enough and $k$ fixed, we have
\begin{equation*}
J_{k-1}(2 \pi x) = x^{-1/2} \left( e (x) U(x) + e(-x) \overline{U}(x)\right), 
\end{equation*} where $U(x)$ is an smooth function defined on $(0, \infty)$ and satisfying $x^{j} U^{(j)}(x) \ll 1$. 
\end{lemma}

\subsection{Stationary phase analysis for exponential integrals}
 This subsection is taken from\cite[Section 8]{Blomer}. In the course of proof of our theorems, we will face  exponential integrals of the form 
$$I = \int w(t) \, e^{ih(t)} \mathrm{d}t,$$
where $w$ is a smooth function supported  in $[a,b]$ and $h$ is a smooth real valued function on $[a,b]$. 
The following lemma will be used to show that  $I$ is negligibly small in the absence of the stationary phase. 
\begin{lemma}
Let $Y \geq 1, X, Q, U, R >0$. And let us further assume that
\begin{itemize}
	\item $w^{(j)}(t) \ll_{j} \frac{X}{U^{j}}$ for $j=0,1,2, \ldots$
	\item $\vert h^{\prime}(t) \vert \geq R$ and $h^{(j)}(t) \ll_{j} \frac{Y}{Q^{j}}$ for $j=2,3,\ldots$
\end{itemize}
Then we have
$$I \ll_{A} (b-a)X \left(\left(\frac{QR}{\sqrt{Y}}\right)^{-A}+\left(RU \right)^{-A} \right).$$
\end{lemma}

The following lemma gives an asymptotic expression for $I$ when the stationary phase exist.
\begin{lemma} \label{stationaryphase}
Let $0 < \delta < 1/10, X, Y, U, Q>0, Z:= Q+X+Y+b-a+1$, and assume that
$$ Y \geq Z^{3 \delta}, \, b-a \geq U \geq \frac{Q Z^{\frac{\delta}{2}}}{\sqrt{Y}}.$$
Assume that $w$  satisfies  
$$w^{(j)}(t) \ll_{j} \frac{X}{U^{j}} \, \, \, \text{for} \,\,  j=0,1,2,\ldots.$$  
Suppose that there exists unique $t_{0} \in [a,b]$ such that $h^{\prime}(t_{0})=0$, and the function $h$ satisfies
$$h^{\prime \prime}(t) \gg \frac{Y}{Q^2}, \, \, h^{(j)}(t) \ll_{j} \frac{Y}{Q^{j}} \, \, \, \, \text{for} \, \, j=1,2,3,\ldots.$$
Then we have
$$I = \frac{e^{i h(t_{0})}}{\sqrt{h^{\prime \prime}(t_{0})}} \, \sum_{n=0}^{3 \delta^{-1}A} p_{n}(t_{0}) + O_{A,\delta}\left( Z^{-A}\right), \, p_{n}(t_{0}) = \frac{\sqrt{2 \pi} e^{\pi i/4}}{n!} \left(\frac{i}{2 h^{\prime \prime}(t_{0})}\right)^{n} G^{(2n)}(t_{0})$$
where 
$$ G(t)=w(t) e^{i H(t)}, \text{and} \, H(t)= h(t)-h(t_{0})-\frac{1}{2} h^{\prime \prime}(t_{0})(t-t_{0})^2.$$
Furthermore, each  $p_{n}$ is a rational function in $h^{\prime}, h^{\prime \prime}, \ldots,$ satisfying the derivative bound
$$\frac{d^{j}}{dt_{0}^{j}} p_{n}(t_{0}) \ll_{j,n} X \left(\frac{1}{U^{j}}+ \frac{1}{Q^{j}}\right) \left( \left(\frac{U^2 Y}{Q^2}\right)^{-n} + Y^{-\frac{n}{3}}\right).$$
\end{lemma}


\section{\bf Application of delta method and summation formulae} \label{aplicircle} 
Let $\pi$   be as defined in  Theorem \ref*{Thm} with $ \lambda_{\pi}(r,n) $ as its Fourier coefficients.
 Let
 \begin{equation} \label{s(x)value}
 S_r(X) = \sum_{n=1}^{\infty} \lambda_{\pi}(r,n) \, e\left(\alpha n^{\beta}\right) V\left(\frac{n}{X}\right).
 \end{equation}
There are two oscillatory terms in the above sum. We will use the delta method to separate these oscillations. Moreover, we will introduce an extra $v$-integral. More precisely, we rewrite the sum $ S_r(X)$ in \eqref{s(x)value} as
  \begin{equation} \label{v-integral}
  S_r(X) = \frac{1}{K} \int_{\mathbb{R}} V\left(\frac{v}{K}\right)\mathop{\sum \sum}_{\substack{n, m=1 \\ n = m } }^{\infty} \lambda_{\pi}(r,n) \, e\left(\alpha m^{\beta}\right) \delta(n-m) \left(\frac{n}{m}\right)^{iv} V\left(\frac{n}{X}\right) U\left(\frac{m}{X}\right)\mathrm{d}v,
  \end{equation}
  where $K=X^{\beta -\eta} < X^{\beta}$ is a parameter which will be chosen later optimally, and $U(x)$ is a smooth function supported in $[1/2,5/2]$, with $U(x)=1$ for $x \in [1,2]$, and $U^{(j)}(x) \ll_{j}1$. The $v$-integral in \eqref{v-integral} serves as a ``conductor lowering mechanism" which is introduced by Munshi in his ground breaking work on  subconvexity bounds for $GL(3)$ $L$-functions in t-aspect \cite{munshi1}.
  
  By repeated integration by parts we see that the $v$-integral in \eqref{v-integral} is negligibly small unless
  $$|n-m| \ll \frac{X}{K} X^{\epsilon}.$$ 
  Therefore we apply the formula for $\delta(n-m)$ given in Lemma \ref{deltasymbol} with $Q = X^{\epsilon}\sqrt{X/K} $. Hence we get
  \begin{align}
  \label{applicacircle}
  S_r(X) &=\frac{1}{QK}\int_{\mathbb{R}}W(x)\int_{\mathbb R}V\left(\frac{v}{K}\right)\sum_{1\leq q\leq Q}\;\frac{g(q,x)}{q}\;\sideset{}{^\star}\sum_{a\bmod{q}} \\
  \nonumber &\times \sum_{m=1}^\infty m^{-iv} e\left(\alpha m^{\beta} \right)e\left(-\frac{am}{q}\right)e\left(-\frac{mx}{qQ}\right)U\left(\frac{m}{X}\right)\\
  \nonumber &\times \mathop{\sum}_{n=1}^\infty \lambda_\pi(r,n)e\left(\frac{an}{q}\right)e\left(\frac{nx}{qQ}\right) n^{iv} V\left(\frac{n}{X}\right) \mathrm{d}v\mathrm{d}x  + O(X^{-2019}).
  \end{align}
  
 \subsection{Poisson summation formula} 
 We now consider the $m$ sum in  \eqref{applicacircle} and proceed to apply the Poisson summation formula. Thus we have
 \begin{align} \label{l-sumpoisson}
 & \sum_{m=1}^\infty m^{-iv} e\left(\alpha m^{\beta} \right)e\left(-\frac{am}{q}\right)e\left(-\frac{mx}{qQ}\right)U\left(\frac{m}{X}\right) \\
 \nonumber = & \sum_{r \, \rm mod \, q} e\left(-\frac{ar}{q}\right) \sum_{l \in \mathbb{Z}} (r+lq)^{-iv} e\left(\alpha(r+lq)^{\beta}\right) e\left(-\frac{(r+lq)x}{qQ}\right) U\left(\frac{r+lq}{X}\right).
 \end{align}
 On applying  the Poisson summation formula to the above $l$-sum, $m$-sum transforms as
 \begin{eqnarray} \label{m-sumafterpoisson}
 X^{1-iv} \sum_{\substack{m \in \mathbb{Z} \\ m \, \equiv a \, \textrm{mod} \, q}} \int_{0}^{\infty} U(y) \, y^{-iv} \, e \left(\alpha \left(Xy\right)^{\beta}- \frac{Xyx}{qQ} - \frac{Xmy}{q}\right) \mathrm{d}y.
 \end{eqnarray}

By repeated integration by parts, we see that the above $y$-integral is negligibly small if  
\begin{equation*} 
\left|m\right| \gg  \max \left\{\frac{qK}{X}, \alpha \beta qX^{\beta-1}, \frac{|x|}{Q} \right\} X^{\epsilon}.
\end{equation*}
Thus the effective range of $m$ is given by 
\begin{equation} \label{m-bound}
|m| \ll  \alpha \beta q X^{\beta -1} X^{\epsilon}= : M_0.
\end{equation}
We summarize the above discussion in the following lemma. 
\begin{lemma}\label{dual m}
Let $M_0$ be as in \eqref{m-bound}. Then we have 
	\begin{align*}
	 \sum_{m=1}^\infty m^{-iv} e\left(\alpha m^{\beta} \right)e\left(-\frac{am}{q}\right)e\left(-\frac{mx}{qQ}\right)&U\left(\frac{m}{X}\right) = X^{1-iv} \sum_{\substack{\left|m\right| \ll M_{0} \\ m \, \equiv a \, \textrm{mod} \, q}} \int_{0}^{\infty} U(y) \, y^{-iv} \\
	 &\times  e \left(\alpha \left(Xy\right)^{\beta}- \frac{Xyx}{qQ} - \frac{Xmy}{q}\right) \mathrm{d}y.
	\end{align*}
\end{lemma}


\subsection{$GL(3)$-Voronoi summation formula} \label{GL3 voro}
As a next step, we consider the $n$-sum in \eqref{applicacircle} and apply  $GL(3)$-Voronoi summation formula. In our setup  $\psi(n)= n^{iv} \, e\left(\frac{nx}{qQ}\right)  V\left(\frac{n}{X}\right)$. Thus on applying Lemma \ref{gl3voronoi} to the $n$-sum we arrive at
 \begin{align} \label{aplicagl3voro}
 & \mathop{\sum}_{n=1}^\infty \lambda_\pi(r,n)e\left(\frac{an}{q}\right)e\left(\frac{nx}{qQ}\right) n^{iv} V\left(\frac{n}{X}\right) \\
 \nonumber & = q \frac{\pi^{-5/2}}{4i} \sum_{\pm} \sum_{n_{1}|qr} \sum_{n_{2}=1}^{\infty}  \frac{\lambda_{\pi}(n_{2},n_{1})}{n_{1} n_{2}} S\left( r\bar{a}, \pm n_{2}; qr/n_{1}\right) \, \Psi_{0,1}^{\pm} \left(\frac{n_{1}^2 n_{2}}{q^3r}\right),
 \end{align}
 where the integral transform $\Psi_{0,1}^{\pm}(x)$ is as given in \eqref{psicombina}.  For further analysis, we only consider the case when the integral transform is $\Psi_{0}(x)$ and we take  + sign in the summation on the right hand side of \eqref{aplicagl3voro}, as the other cases can be dealt similarly. Furthermore, we can also assume that $ n_{1}^2 n_{2} X /q^3r \gg X^{\epsilon}$, since in the complimentary range, we get the desirable bound, namely, 
 \begin{equation} \label{comlimentarybound}
 S_{\text{compli}}(X) \ll X^{\frac{3}{4}-\frac{3 \beta}{4}+\frac{7 \eta}{4}}.
 \end{equation}
It can be seen as follows: by taking the $v$-integral into the explicit expression of $\Psi_{0}(x)$ we get a restriction on the $z$-variable ($|z-y| \ll \frac{1}{K} X^{\epsilon}$). For the integral over the  vertical line,  we use Stirling approximation formula for the Gamma function. 

In the case  $n_{1}^2 n_{2} X/q^3r \gg X^{\epsilon}$, we use the asymptotic behaviour of the function $\Psi_{0}(n_{1}^2n_{2}/q^3r)$ given in Lemma \eqref{GL3oscilation}. Taking $J$  large enough,  we can ignore the error term in \eqref{GL3oscilation}, and hence we only consider the leading term as other terms can be dealt similarly. 

Keeping the above discussion in mind, on applying Lemma \eqref{GL3oscilation} to \eqref{aplicagl3voro}, we see that the $n$-sum is given by
\begin{align} 
\frac{X^{2/3+iv}}{qr^{2/3}} \sum_{n_1|qr}n_1^{1/3}&\sum_{n_2=1}^\infty \frac{\lambda_\pi(n_1,n_2)}{n_2^{1/3}}S(r\bar a,  n_2; qr/n_1)\\
\nonumber &\times \int_0^\infty V(z)z^{iv}e\left(\frac{Xxz}{qQ}\pm \frac{3(Xn_1^2n_2z)^{1/3}}{qr^{1/3}}\right)\mathrm{d}z.
\end{align}

By repeated integration by parts we see that the $z$-integral is negligibly small if
\begin{equation} \label{n-sumbound}
n_{1}^{2} n_{2} \gg  \max \left\{\frac{\left(qK\right)^{3}r}{X}, K^{3/2} X^{1/2} rx^{3} \right\} X^{\epsilon}=:N_0. 
\end{equation}
We end this subsection by recording the above arguments in the following lemma.
\begin{lemma}\label{first gl3 voronoi}
 Let $\psi(n)= n^{iv} \, e\left(\frac{nx}{qQ}\right)  V\left(\frac{n}{X}\right)$ and $N_0$ be as in \eqref{n-sumbound}. Then we have 
 \begin{align*}
  \mathop{\sum}_{n=1}^\infty \lambda_\pi(r,n)e\left(\frac{an}{q}\right)\psi(n)=&\frac{X^{2/3+iv}}{qr^{2/3}} \sum_{n_1|qr}n_1^{1/3}\sum_{n_2\ll N_0/n_1^2}\frac{\lambda_\pi(n_1,n_2)}{n_2^{1/3}}S(r\bar a,  n_2; qr/n_1)\\
  \nonumber &\times \int_0^\infty V(z)z^{iv}e\left(\frac{Xxz}{qQ}\pm \frac{3(Xn_1^2n_2z)^{1/3}}{qr^{1/3}}\right)\mathrm{d}z.
 \end{align*}
 \end{lemma}

\section{\bf Simplifying the integrals} \label{4 integral}
On applying Lemma \ref{dual m} and Lemma\ref{first gl3 voronoi} to \eqref{applicacircle}, we see that 
 $S_r(X)$, after executing the $a$-sum, is  given by
\begin{align} \label{essentialS(X)}
&  \frac{X^{5/3}}{QKr^{2/3}}\int_{\mathbb{R}}W(x)\int_{\mathbb R}V\left(\frac{v}{K}\right)\sum_{1\leq q\leq Q}\;\frac{g(q,x)}{q^2} \\
\nonumber &\times  \sum_{\substack{\left|m\right| \ll M_{0} \\ (m,q)=1}} \int_{0}^{\infty} U(y) \, y^{-iv} \, e \left(\alpha \left(Xy\right)^{\beta}- \frac{Xxy}{qQ} - \frac{Xmy}{q}\right) \mathrm{d}y \, \, \sum_{n_1|qr}n_1^{1/3}\\
\nonumber &\times  \sum_{n_2 \ll \frac{N_{0}}{n_{1}^2}} \frac{\lambda_\pi(n_1,n_2)}{n_2^{1/3}}S(r m,  n_2; qr/n_1) \int_0^\infty V(z)z^{iv}e\left(\frac{Xxz}{qQ}\pm \frac{3(Xn_1^2n_2z)^{1/3}}{qr^{1/3}}\right)\mathrm{d}z \mathrm{d}x  \mathrm{d}v.
\end{align}
In this section, we will  simplify the four-fold integrals in \ref{essentialS(X)}. Let us first consider the $x$-integral, which is given by
$$\int_{\mathbb R} W(x)\,g(q,x) \,e\left(\frac{Xx(z-y)}{qQ}\right) \, \mathrm{d}x.$$
Using the property of $g(q,x)$ given in Subsection \ref{circlemethod}, we see that the above integral splits as
$$\int_{\mathbb R} W(x) \,e\left(\frac{Xx(z-y)}{qQ}\right) \, \mathrm{d}x + \int_{\mathbb R} W(x)\,h(q,x) \,e\left(\frac{Xx(z-y)}{qQ}\right) \, \mathrm{d}x.$$
In the first integral, by repeated integration by parts we see that the integral is negligibly small unless $|z-y| \ll X^{\epsilon } q/QK $. In the second integral case, we get a weaker restriction $|z-y|\ll X^{\epsilon}/K$ by considering $v$-integral. But we have a better bound  for $h(q,x)$,   $h(q,x) \ll 1/qQ$. As a result, we get better bounds in this case. We will continue our analysis with the first integral.

Letting $z=y+u$ with $|u|\ll X^{\epsilon} q /QK$ in the $z$-integral in \eqref{essentialS(X)}, we see that the $y$-integral changes into
\begin{equation} \label{change-y-integral}
\mathcal{I} \left(m,n_{1}^{2}n_{2},q \right) := \int_{0}^{\infty} U_{v,u}(y) e\left( \alpha \left(Xy\right)^{\beta} - \frac{Xmy}{q}\pm \frac{3(Xn_1^2n_2\left(y+u\right))^{1/3}}{qr^{1/3}} \right)\mathrm{d}y,
\end{equation}
where $U_{v,u}(y)=U(y)V(y+u)(1+u/y)^{iv}$. We note that $U_{v,u}^{(j)}(y)\ll X^{\epsilon}$. From now onwards, by the abuse of notations, we will denote it by $U(y)$. Hence the  four-fold integral in \eqref{essentialS(X)} looks like
$$\int_{\mathbb{R}}W(x) \int_{\mathbb{R}}V\left(\frac{v}{K}\right)\int_{\mathbb{R}}e\left(\frac{Xxu}{qQ}\right)\mathcal{I} \left(m,n_{1}^{2}n_{2},q \right)\mathrm{d}u\mathrm{d}v\mathrm{d}x.$$
We will estimate $x$, $v$ and $u$-integral trivially later and we will see that estimates for $\mathcal{I} \left(m,n_{1}^{2}n_{2},q \right)$ are uniform with respect to $u$ and $v$. Hence we have reduced the four fold integral into 
$$K \times \frac{q}{QK}\times  \mathcal{I} \left(m,n_{1}^{2}n_{2},q \right).$$
\subsection{Estimates for the $y$-integral $\mathcal{I}$} 
In this subsection, we will estimate the $y$-integral $\mathcal{I} \left(m,n_{1}^{2}n_{2},q \right)$. More precisely, we have the following lemma.
\begin{lemma} \label{L^2bound}
Let $W$ be a smooth bump function supported in $[1,2]$ and $W^{(j)}\ll 1$. Then we have
$$\mathcal{W} := \int_{\mathbb{R}} W(w) \vert \mathcal{I}(m,N_{0}w^3,q)\vert^{2} \mathrm{d}w \ll \frac{X^{\epsilon}}{X^{\beta}}.$$
\end{lemma} 
\begin{proof}
Consider the $y$-integral
$$\mathcal{I} \left(m,N_{0}w^3,q \right) = \int_{0}^{\infty} U(y) e\left( \alpha \left(Xy\right)^{\beta} - \frac{Xmy}{q}\pm \frac{3(XN_{0}\left(y+u\right))^{1/3} w}{qr^{1/3}} \right)\mathrm{d}y.$$
The phase function of the above integral is given by
$$f(y)=\alpha \left(Xy\right)^{\beta} - \frac{Xmy}{q}\pm \frac{3(XN_{0}\left(y+u\right))^{1/3} w}{qr^{1/3}}.$$
On computing the higher order derivatives, we see that 
$$f^{\prime \prime}(y)= \alpha \beta (\beta-1) X^{\beta} y^{\beta -2} \mp \frac{2}{3} \frac{(XN_{0})^{1/3}w}{q r^{1/3}y^{5/3}} + O\left(\frac{(XN_0)^{1/3}|u|}{qr^{1/3}} \right).$$
For this to be smaller then $X^{\beta}$ in magnitude one at least needs a negative sign in the second term and $(XN_{0})^{1/3}w/qr^{1/3} \asymp X^{\beta}$. Except this case,  by using second derivative bound, we get 
$$\mathcal{I} \left(m,N_{0}w^3,q \right) \ll \frac{1}{ \sqrt{X^{\beta}}}.$$
In  the special situation, i.e., when  $(XN_{0})^{1/3}w/qr^{1/3}\asymp X^{\beta}$ and there is a negative sign in the second term, opening the absolute square and interchanging the integration symbols, we see that 

\begin{align}
\mathcal{W} &\ll \int \int U(y_{1}) U(y_{2}) \left| \int W(w) e \left(\frac{3w(XN_{0})^{1/3}}{qr^{1/3}}\left((y_{1}+u)^{1/3}-(y_{2}+u)^{1/3}\right)\right) \, \mathrm{d}w\right| \nonumber\\
\nonumber& \ll \int \int _{\substack{|y_{1}-y_{2}|\ll \frac{X^{\epsilon}}{X^{\beta}}}} U(y_{1}) U(y_{2}) \, \mathrm{d}y_{1} \, \mathrm{d}y_{2} + O(X^{-2019})\ll \frac{1}{X^{\beta}}.
\end{align}
Hence the lemma follows.
\end{proof}
 
\section{\bf Cauchy-Schwarz and Poisson}
\subsection{Cauchy inequality}
 After simplifying the integrals, the expression  in \eqref{essentialS(X)} has essentially reduced to 
\begin{align} \label{setting-for-cauchy}
\frac{X^{2/3}}{r^{2/3}} & \sum_{1 \leq q \leq Q} \frac{1}{q}   \sum_{n_{1}|qr} n_{1}^{1/3} \sum_{n_{2} \ll N_{0}/n_{1}^{2}} \frac{\lambda_{\pi}(n_{1},n_{2})}{n_{2}^{1/3}} \\
\nonumber & \times \sum_{\substack{\left|m\right| \ll M_{0} \\ (m,q)=1}}  S(r m,  n_2; qr/n_1) \mathcal{I}\left(m,n_{1}^{2}n_{2},q \right),
\end{align}
where $\mathcal{I}\left(m,n_{1}^{2}n_{2},q \right)$ is as given in \eqref{change-y-integral}. 

Spliting the sum over $q$ in dyadic blocks $q\backsim C$ and  writing $q=q_1q_2$ with $q_1|(n_1r)^\infty$, $(n_1r,q_2)=1$, and apply the Cauchy-Schwarz inequality over $n_{2}$-sum.  We see that \eqref{setting-for-cauchy} is bounded by
\begin{equation} 
 X^{\epsilon} \sup_{C \leq Q} \, \frac{X^{2/3} N_{0}^{1/6}}{C} \sum_{\frac{n_1}{(n_1,r)} \ll C} n_1^{1/3}\Theta^{1/2}\sum_{\frac{n_1}{(n_1,r)}|q_1|(n_1r)^\infty}\Omega^{1/2},
\end{equation}
 where 
\begin{equation} \label{omegavalue1}
\Omega = \sum_{n_{2} \ll N_0/n_1^2} \Big|\sum_{q \sim C} \sum_{\substack{\left|m\right| \ll M_{0}\\ (m,q)=1}} S\left(r{m}, n_{2};qr/n_{1}\right) \, \mathcal{I}\left(m,n_{1}^{2}n_{2},q \right) \Big|^{2},
\end{equation}
and 
\begin{align}\label{theta}
\Theta=\sum_{n_2\ll N_0/n_1^2} \frac{|\lambda_\pi(n_1,n_2)|^2}{n_2^{2/3}}.
\end{align} 
\subsection{Poisson summation}\label{Poisson}
We now smooth out the outer sum in \ref{omegavalue1} with an appropriate bump function, say, $W$ to apply the Poisson summation formula.  Thus 
\begin{equation*} 
\Omega = \sum_{n_{2} \in \mathbb{Z}} W\left(\frac{n_{1}^{2}n_{2}}{N_{0}}\right) \Big|\sum_{q \sim C} \sum_{\substack{\left|m\right| \ll M_{0}\\ (m,q)=1}} S\left(r{m}, n_{2};qr/n_{1}\right) \, \mathcal{I}\left(m,n_{1}^{2}n_{2},q \right) \Big|^{2}.
\end{equation*}
Opening the absolute value square  we see that
\begin{equation} 
\Omega = \mathop{\sum \sum}_{q,q' \sim C} \mathop{\sum \sum}_ {\substack{\left|m\right|,\left|m'\right| \ll M_{0} \\ (m,q)=(m',q')=1 }} \mathcal{L} \left(m,m',q,q',n_{1}\right),
\end{equation}
where $\mathcal{L} \left(m,m',q,q',n_{1}\right)$ is given by
\begin{equation*} \label{L-value}
\sum_{n_{2}\in \mathbb{Z}} W\left(\frac{n_{1}^2 n_{2}}{N_{0}}\right) S\left(r{m},n_{2};qr/n_{1}\right) S\left(r{m'},n_{2};q'r/n_{1}\right) \mathcal{I} \left(m,n_{1}^2 n_{2},q\right) \overline{\mathcal{I} \left(m',n_{1}^2 n_{2},q'\right)}, 
\end{equation*}
and $q^{\prime}=q_1q_2^{\prime}$. Using the change of variable $n_2 \rightsquigarrow n_2q_1q_2q_2^\prime r/n_1+\gamma$ with $ 0 \leq \gamma < q_1q_2q_2^\prime r/n_1$ we arrive at
\begin{align}
\sum_{\gamma \, \rm mod \, q_1q_2q_2^{\prime}/n_1}&   S\left(r{m},\gamma;qr/n_{1}\right)  S\left(r{m^{\prime}},\gamma;q^{\prime}r/n_{1}\right)\sum_{n_2 \in \mathbb{Z}} \mathcal{I} \left(m,n_{1}^2 (\gamma+n_2q_1q_2q_2^\prime r/n_1),q\right) \notag \\
 & \times  \overline{\mathcal{I} \left(m^{\prime},n_{1}^2 (\gamma+n_2q_1q_2q_2^\prime r/n_1),q^{\prime}\right)} \, W\left(\frac{n_{1}^2 (\gamma+n_2q_1q_2q_2^\prime r/n_1)}{N_{0}}\right).
\end{align}
 On applying the Poisson summation formula to the $n_2$-sum we get that
\begin{equation} \label{final-l-value}
\mathcal{L} \left(m,m',q,q',n_{1}\right) = \frac{N_{0}}{ q_1q_2q_2^{\prime}rn_1} \sum_{n_{2} \in \mathbb{Z}} \mathcal{C} \, \mathcal{J}
\end{equation}
where $	\mathcal{C}:=\mathcal{C}(m,m',q,q',n_{1},n_{2})$ is given by
\begin{align}\label{C}
 	\mathcal{C}= \sum_{\gamma \bmod \frac{q_1q_2q_2^{\prime}r}{n_1}} S\left(r{m},\gamma ;qr/n_{1}\right) S\left(r{m^{\prime}},\gamma;q^{\prime}r/n_{1}\right) e\left(\frac{ \gamma n_1n_2}{ q_1q_2q_2^{\prime}r}\right)
\end{align}
and $\mathcal{J}=\mathcal{J}(m,m',q,q',n_{1},n_{2})$ is defined as
\begin{align}\label{J}
	\mathcal{J} = \int W\left(w\right) \mathcal{I} \left(m,N_{0}w,q\right) \; \overline{\mathcal{I} \left(m',N_{0}w,q'\right)} e\left(-\frac{N_0 n_2w}{q_1q_2q_2^\prime r n_1}\right) \;  \mathrm{d}w.
\end{align}
By repeated integration by parts we see that the integral $\mathcal{J}$ is negligibly small unless 
\begin{align}\label{n_2 dual}
	n_{2} \ll  \frac{\sqrt{XK} Cn_1r}{N_{0}q_1} X^{\epsilon}:= \tilde{N}.
\end{align}
We conclude this section by recording the above discussion in the following lemma.
\begin{lemma}\label{SX after poisson}
	  Let $\mathcal{C}$ and $\mathcal{J}$ be as in \eqref{C} and \eqref{J} respectively. Let $\tilde{N}$ be as in \eqref{n_2 dual}. Then we have 
	\begin{equation} \label{aftercauchy1}
S_r(X)\ll X^{\epsilon} \sup_{C \leq Q} \, \frac{X^{2/3} }{C} \sum_{\frac{n_1}{(n_1,r)} \ll C} n_1^{1/3}\Theta^{1/2}\sum_{\frac{n_1}{(n_1,r)}|q_1|(n_1r)^\infty}\Omega^{1/2},
	\end{equation}
	where 
	\begin{align}\label{omegavalue2}
		\Omega=\frac{N_{0}}{ q_1q_2q_2^{\prime}rn_1}\mathop{\sum \sum}_{q,q' \sim C} \mathop{\sum \sum}_ {\substack{\left|m\right|,\left|m'\right| \ll M_{0} \\ (m,q)=(m',q')=1 }} \sum_{n_{2} \ll \tilde{N}} \mathcal{C} \, \mathcal{J},
	\end{align}
	with 
	$$M_0 \ll |\alpha|\beta CX^{\beta -1}, \ \ \ \mathrm{and} \ \ \ N_0 \ll K^{3/2}X^{1/2}r.$$
\end{lemma}
\section{\bf Analysis of the character sum}
In this section, we will analyze the character sum $\mathcal{C}$ given in \eqref{C}. A similar character was treated by Munshi in \cite[Section 6]{munshi3}. We have the following lemma. 
\begin{lemma} \label{charsum}
Let $\mathcal{C}=\mathcal{C}(m,m',q,q',n_{1},n_{2})$ be as in \eqref{C}. Then we have 
$$\mathcal{C}(m,m',q,q',n_{1},n_{2}) \ll  \frac{q_1^2r^2q_2q_2^{\prime}}{n_1^2}(q_2,q_2^{\prime}, n_2).$$
If $n_{2}=0$, Then
\begin{align*}
\mathcal{C}=\frac{q_1q_2^2r}{n_1}\mathop{\sideset{}{^ \star}\sum_{\alpha \, {\rm mod} \, qr/n_1 }} e\left(\frac{ r n_1\alpha({m} -{m^{\prime}})}{ q_1q_2r}\right)
\end{align*}
\end{lemma}
\begin{proof}
Let's recall that
\begin{align*}
\mathcal{C} = \sum_{\gamma \bmod \frac{q_1q_2q_2^{\prime}r}{n_1}} S\left(r{m},\gamma ;qr/n_{1}\right) S\left(r{m^{\prime}},\gamma;q^{\prime}r/n_{1}\right) e\left(\frac{ \gamma n_1n_2}{ q_1q_2q_2^{\prime}r}\right).
\end{align*}
On expanding the Kloosterman sum and then executing the $\gamma$-sum we get that
\begin{align}\label{C1}
	\mathcal{C} =\frac{q_1q_2q_2^{\prime}r}{n_1}\mathop{\sideset{}{^ \star}\sum_{\substack{\alpha \, {\rm mod} \, qr/n_1 }} \  \sideset{}{^ \star}\sum_{\substack{\alpha^\prime \, {\rm mod} \, q^\prime r/n_1 }}}_{ \bar{\alpha}q_2^\prime -\bar{\alpha}^\prime q_2\equiv -n_2  \, {\rm mod} \, q_1q_2q_2^\prime r/n_1} e\left(\frac{ r n_1({m}\alpha q_2^{\prime}-{m^{\prime}}\alpha^{\prime}q_2)}{ q_1q_2q_2^{\prime}r}\right).
\end{align}
If $n_2=0$, then the above congruence condition implies that $q_2=q_2^{\prime}$ and $\alpha=\alpha^{\prime}$. And hence
\begin{align*}
\mathcal{C}=\frac{q_1q_2^2r}{n_1}\mathop{\sideset{}{^ \star}\sum_{\alpha \, {\rm mod} \, qr/n_1 }} e\left(\frac{ r n_1\alpha({m} -{m^{\prime}})}{ q_1q_2r}\right).
\end{align*}
Hence the second part of the lemma follows.
Now let $n_2 \neq 0$. We dominate $\mathcal{C}$ in \eqref{C1} by a product of $\mathcal{C}_1$ and $\mathcal{C}_2$, where 
\begin{align*}
	\mathcal{C}_1=\frac{q_1r}{n_1}\mathop{\sideset{}{^ \star}\sum_{\substack{\alpha \, {\rm mod} \, q_1r/n_1 }} \  \sideset{}{^ \star}\sum_{\substack{\alpha^\prime \, {\rm mod} \, q_1 r/n_1 }}}_{ \bar{\alpha}q_2^\prime -\bar{\alpha}^\prime q_2\equiv -n_2  \, {\rm mod} \, q_1 r/n_1}1,
\end{align*}
and 
\begin{align*}
		\mathcal{C}_2 =q_2q_2^{\prime}	\mathop{\sideset{}{^ \star}\sum_{\substack{\alpha \, {\rm mod} \, q_2 }} \  \sideset{}{^ \star}\sum_{\substack{\alpha^\prime \, {\rm mod} \, q_2^\prime }}}_{ \bar{\alpha}q_2^\prime -\bar{\alpha}^\prime q_2\equiv -n_2  \, {\rm mod} \, q_2q_2^\prime}1.
\end{align*}
In the first sum $\mathcal{C}_1$ the congruence condition determine $\alpha$ uniquely in terms of $\alpha^\prime$. Hence 
\begin{align}\label{C1 bound}
\mathcal{C}_1 \ll \frac{(q_1r)^2}{n_1^2}.
\end{align}
In the second sum $\mathcal{C}_2$, given any $\alpha \ {\rm mod}  \ q_2 $, we observe that $\alpha^{\prime}$ is  determined uniquely modulo $q_2^{\prime}$. Moreover, reducing the congruence modulo $q_2$, we observe that number of such $\alpha$'s is given by $(q_2, q_2^{\prime},n_2)$. And hence 
\begin{align}\label{C2 bound}
	\mathcal{C}_2 \ll q_2q_2^{\prime}(q_2,q_2^{\prime},n_2).
\end{align} 
Combining \eqref{C1 bound} and \eqref{C2 bound}, we get the lemma.

\end{proof} 
\section{\bf Analysis of  zero frequency}\label{zero freq}
With all the ingredients in hand, we will now  estimate $S_r(X)$ in  the present and coming sections. We start by considering $\Omega$ in Lemma \ref{SX after poisson}. Let $\Omega_{0}$  denotes the part  of $\Omega$ corresponding to $n_2=0$ and let $S_r^{0}(X)$ denotes the part of $S_r(X)$ corresponing to $\Omega_{0}$. We will prove the following lemma in this section.

 \begin{lemma}\label{S(zero)}
 We have 
 \begin{equation} \label{zerocontribution}
 S_r^{0}(X) \ll \sqrt{\alpha \beta}r^{7/6}X^{\frac{3 \beta}{4}+\frac{3}{4}-\frac{3 \eta}{4}}.
 \end{equation} 
\end{lemma}
\begin{proof}
For $n_{2}=0$, we have seen in  Lemma \ref{charsum} that    $q_2=q_2^{\prime}$. Also using Lemma \ref{charsum} and Lemma \ref{L^2bound} in  \eqref{final-l-value} we get 
$$\mathcal{L}(m,m',q,n_{1}) \ll \frac{N_{0}}{n_{1}^2 X^{\beta}} \, \Big| \mathop{\sideset{}{^ \star}\sum_{\alpha \, {\rm mod} \, qr/n_1 }} e\left(\frac{ r n_1\alpha({m} -{m^{\prime}})}{ q_1q_2r}\right) \Big|.$$

 By substituting the above bound  in \eqref{omegavalue2}, we see that $\Omega_{0}$ is  dominated by
\begin{align}
 \frac{N_{0}}{X^{\beta} n_{1}^2}&  \sum_{q_2 \sim C/q_1} \sum_{\substack{\left|m\right|\ll M_{0} \\ (m,q)=1}}  \frac{qr}{n_{1}}+ \frac{N_{0}}{X^{\beta}n_{1}^2} \sum_{q_2 \sim C/q_1} \mathop{\sum \sum}_{\substack{\left|m\right|,\left|m'\right| \ll M_{0} \\ m \neq m' \\ (m,q)=1 = (m',q)=1}}\Big| \mathop{\sideset{}{^ \star}\sum_{\alpha \, {\rm mod} \, qr/n_1 }} e\left(\frac{ r n_1\alpha({m} -{m^{\prime}})}{ q_1q_2r}\right) \Big| \notag \\
& \ll  \frac{N_{0} M_{0} C^{2}r}{X^{\beta} n_{1}^3q_1}+\frac{N_{0}}{X^{\beta}n_{1}^2} \sum_{q_2 \sim C/q_1} \mathop{\sum \sum}_{\substack{\left|m\right|,\left|m'\right| \ll M_{0} \\ m \neq m' \\ (m,q)=1 = (m',q)=1}}(r(m-m^{\prime}),qr/n_1) \notag \\
& \ll  \frac{N_{0} M_{0} C^{2}r}{X^{\beta} n_{1}^3q_1}+\frac{N_0M_0^2rC}{X^{\beta}n_1^2q_1}.
\end{align}

 Now  substituting the above bound in place of $\Omega$ in $S_r(X)$ in Lemma \ref{SX after poisson}, we get  that
 \begin{align*}
 	\mathcal{S}_{r}^{0}(X) \ll \sup_{C \leq Q} \, \frac{X^{2/3} }{C} \sum_{\frac{n_1}{(n_1,r)} \ll C} n_1^{1/3}\Theta^{1/2}\sum_{\frac{n_1}{(n_1,r)}|q_1|(n_1r)^\infty}\frac{\sqrt{N_0M_0Cr}}{X^{\beta/2}n_1q_1^{1/2}}\left(\sqrt{C}+\sqrt{M_0}\right),
 \end{align*}
 Executing the $q_1$-sum trivially, we get
  \begin{align*}
 \mathcal{S}_{r}^{0}(X) \ll \sup_{C \leq Q}  \frac{X^{2/3} }{C} \sum_{\frac{n_1}{(n_1,r)} \ll C} \frac{(n_1,r)^{1/2}}{n_1^{7/6}}\Theta^{1/2}\frac{\sqrt{N_0M_0Cr}}{X^{\beta/2}}\left(\sqrt{C}+\sqrt{M_0}\right),
 \end{align*}
 We evaluate $n_1$-sum, using the Cauchy's inequality  and the Ramanujan bound (see Lemma \ref{ramanujan}), as
 \begin{align}\label{theta bound}
 \sum_{n_1\ll Cr}\frac{(n_1,r)^{1/2}}{n_1^{7/6}}\Theta^{1/2} \ll \left[\sum_{n_1 \ll Cr}\frac{(n_1,r)}{n_1}\right]^{1/2}\left[\mathop{\sum \sum}_{n_{1}^{2} n_{2} \leq N_0} \frac{\vert \lambda_{\pi}(n_{1},n_{2})\vert ^{2} }{(n_1^2n_2)^{2/3}}\right]^{1/2}\ll \,N_0^{1/6}.
 \end{align}
 Thus,  using the fact  $C \gg M_0$, we have 
 \begin{align*}
 	 \mathcal{S}_{r}^{0}(X) \ll \sup_{C \leq Q} \frac{X^{2/3}N_0^{2/3} }{C} \frac{\sqrt{M_0Cr}}{X^{\beta/2}}\left(\sqrt{C}\right) \ll  \sup_{C \leq Q} X^{2/3-\beta/2}N_0^{2/3}(M_0r)^{1/2},
 \end{align*}
 Finally plugging in the bounds  $M_0 \ll \alpha \beta C X^{\beta -1} $, $N_0 \ll K^{3/2}X^{1/2}r$ and $C \leq Q=X^{(1-\beta+\eta)/2}$ gives us the lemma. 
%
%
\end{proof}
 
 \section{\bf Analysis of  non-zero frequencies}
 \subsection{Estimates for small  $q \sim C$}
In this section, we will consider the cases which are compliment to Section \ref{zero freq}. Let $\Omega_{\neq 0}$ denotes the part of $\Omega$ in \eqref{omegavalue2} which is compliment to $\Omega_0$ and let $S_r^{\neq 0}(X)$ denotes the part of $S_r(X)$ in \eqref{aftercauchy1} which corresponds to  $\Omega_{\neq 0}$. Firstly, we consider the case when   $C \ll  X^{\frac{1}{2}-\frac{\beta}{2}+ \epsilon}$. We have the following lemma. 
 \begin{lemma}\label{S(small)}
Let $S_r^{\neq 0}(X)$  be as above. Let $S_{r, small}^{\neq 0}(X)$ denotes the contribution of $C \ll  X^{\frac{1}{2}-\frac{\beta}{2}+ \epsilon}$ to $S_r^{\neq 0}(X)$. Then we have
 \begin{equation} \label{smallmodulus}
 S_{r, small}^{\neq 0}(X) \ll |\alpha|\beta r^{7/6}X^{\frac{3}{4}+ \frac{\beta}{4}-\frac{\eta}{2}}.
 \end{equation}
 \end{lemma}

\begin{proof}
 Recall from Lemma \ref{SX after poisson} that to bound $S_r^{\neq 0}(X)$, we have to estimate $\Omega_{\neq 0}$ first, which is given as 
\begin{align}\label{omega nonzero starting}
\Omega_{\neq 0} \ll \frac{N_{0}}{ q_1q_2q_2^{\prime}rn_1}\mathop{\sum \sum}_{q_2,q_2^{\prime} \sim C/q_1} \mathop{\sum \sum}_ {\substack{\left|m\right|,\left|m'\right| \ll M_{0} \\ (m,q)=(m',q')=1 }} \sum_{\substack{\left|n_{2}\right| \ll \tilde{N} \\ n_{2} \neq 0}} \mathcal{C} \, \mathcal{J} .
\end{align}
Using Lemma \ref{L^2bound} for $\mathcal{J}$ and  Lemma \ref{charsum} for $\mathcal{C}$, we infer that
\begin{align*}
	\Omega_{\neq 0}&\ll\frac{N_{0}q_1r}{ X^{\beta}n_1^3}\mathop{\sum \sum}_{q_2,q_2^{\prime} \sim C/q_!} \mathop{\sum \sum}_ {\substack{\left|m\right|,\left|m'\right| \ll M_{0} \\ (m,q)=(m',q')=1 }} \sum_{\substack{\left|n_{2}\right| \ll \tilde{N} \\ n_{2} \neq 0}}(q_2,q_2^{\prime}, n_2)\\
	& \ll \frac{N_{0}\tilde{N}q_1r}{ X^{\beta}n_1^3}\mathop{\sum \sum}_{q_2,q_2^{\prime} \sim C/q_1} \mathop{\sum \sum}_ {\substack{\left|m\right|,\left|m'\right| \ll M_{0} \\ (m,q)=(m',q')=1 }}1
\end{align*}
Executing the remaining sums trivially, we get that
\begin{align}\label{nonzero omega bound}
	\Omega_{\neq 0}\ll\frac{N_0\tilde{N}C^2M_0^2r}{X^\beta n_1^3q_1}\ll \frac{C^3M_0^2r^2(XK)^{1/2}}{X^{\beta}n_1^2q_1^2}. 
\end{align}
In the last inequality we have used $\tilde{N}$ bound given in \eqref{n_2 dual}.
Finally, plugging in the above bound in \eqref{aftercauchy1}, we get
\begin{equation} 
S_{r, small}^{\neq 0}(X)\ll \sup_{C \leq X^{(1-\beta)/2}} \, \frac{X^{2/3} }{C} \sum_{\frac{n_1}{(n_1,r)} \ll C} n_1^{1/3}\Theta^{1/2}\sum_{\frac{n_1}{(n_1,r)}|q_1|(n_1r)^\infty}\left(\frac{C^3M_0^2r^2(XK)^{1/2}}{X^{\beta}n_1^2q_1^2}\right)^{1/2},
\end{equation}
Note that 
\begin{align*}
	\sum_{\frac{n_1}{(n_1,r)} \ll C} n_1^{1/3}\Theta^{1/2}\sum_{\frac{n_1}{(n_1,r)}|q_1|(n_1r)^\infty}\frac{1}{q_1n_1}&\ll 	\sum_{\frac{n_1}{(n_1,r)} \ll C} \frac{(n_1,r)}{n_1^{5/3}}\Theta^{1/2}\\
	&\ll \sum_{\frac{n_1}{(n_1,r)} \ll C} \frac{(n_1,r)^{1/2}}{n_1^{7/6}}\Theta^{1/2}\ll N_0^{1/6}.
\end{align*}
Last inequality follows from \eqref{theta bound}. And Hence 
\begin{align} \label{S(small bound)}
S_{r, small}^{\neq 0}(X) &\ll \sup_{C \leq X^{(1-\beta)/2}} \, \frac{X^{2/3}N_0^{1/6} }{C} \left(\frac{C^3M_0^2r^2(XK)^{1/2}}{X^{\beta}}\right)^{1/2} \notag\\
&\ll \sup_{C \leq X^{(1-\beta)/2}}X^{11/12-\beta/2}N_0^{1/6}C^{1/2}M_0rK^{1/4}.
\end{align}
Finally using bounds for $M_0$, $N_0$ and $C$, we get the lemma.
%
\end{proof}

\subsection{Estimates for generic $q \sim C $}
It now remains to tackle the case when $q \sim C \gg  X^{(1-\beta)/2}$. In this case, we need a better bound for the integral $\mathcal{J}$ in \eqref{J}.
 We have the following lemma.
 
 \begin{lemma}\label{J better}
 For $C \gg 1/2-\beta/2$, we have
 $$\mathcal{J}=\mathcal{J}(m,m',q,q',n_{1},n_{2}) \ll  \frac{1}{X^{\beta}} \frac{C^{2/3}r^{1/3}n_1^{1/3}}{n_2^{1/3}q_1^{1/3}N_0^{1/3}}.$$
 \end{lemma}
 
 \begin{proof}
Let's recall from \eqref{J} that 
\begin{align*}
\mathcal{J} = \int W\left(w\right) \mathcal{I} \left(m,N_{0}w,q\right) \; \overline{\mathcal{I} \left(m',N_{0}w,q'\right)} e\left(-\frac{N_0 n_2w}{q_1q_2q_2^\prime r n_1}\right) \;  \mathrm{d}w,
\end{align*}
 where
 $$I \left(m,N_{0} w,q \right) = \int_{0}^{\infty} U(y) e\left( \alpha \left(Xy\right)^{\beta} - \frac{Xmy}{q}\pm \frac{3 \left(XN_{0}w\left(y+u\right)\right)^{1/3}}{qr^{1/3}} \right)\mathrm{d}y. $$
Let
 $$A= \frac{Xm}{q} \quad \text{and} \;  B= \frac{3 \left(X N_{0}w\right)^{1/3}}{qr^{1/3}}.$$
 Firstly, we consider  the term $e(\pm B(y+u)^{1/3})$ in $I(m,N_0w,q)$, which can be expanded as
 $$e(\pm B(y+u)^{1/3})=e(\pm By^{1/3})e(\pm Bu/3y^{-2/3}+...).$$ 
 We observe that the second term of the right hand side does not oscillate with respect to $y$. So we can insert it into the weight function. 

Thus the phase function in the exponential integral $I(...)$ is essentially given by
 $$f(y)= \alpha X^{\beta} y^{\beta} - A y \pm B y^{1/3}.$$
 Note that for $q \sim C \gg X^{\frac{1}{2}-\frac{\beta}{2}+ \epsilon}$, we have $B \ll X^{\beta - \frac{\eta}{2}}$ and  $X^{\beta} \gg A \gg X^{\frac{\beta+1}{2} -\frac{\eta}{2}} \gg X^{\beta - \eta /2}$ as $\beta <1$. Moreover, if $m$ is negative or $A \ll X^{\beta-\epsilon}$ then the $y$-integral is negligibly small. Hence from now onwards, we assume that $m >0$ and $A \asymp X^{\beta}$.
 Next, we will find the stationary point, say, $y_{\text{stat}}$ of the above phase function. In fact, it can be written as
 $$y_{\text{stat}} = y_{0}+y_{1}+y_2+ \ldots,$$
 where $y_{k}= h_{k}(X,A) \left(B/X^{\beta}\right)^{k}$ with $h_{k}(X,A) \ll 1$ for  $k \geq 0$.   Explicit calculation yields
 $$y_{0}= \left(\frac{\alpha \beta X^{\beta}}{A}\right)^{\frac{1}{1- \beta}}, \quad y_{1}= \frac{\mp B}{3 \alpha \beta \left(\beta -1\right) X^{\beta}} \left(\frac{\alpha \beta X^{\beta}}{A}\right)^{\frac{4-3 \beta}{3 \left(1-\beta\right)}} .$$
 Therefore, using stationary phase analysis Lemma \ref{stationaryphase}, the integral $I(...)$ is essentially given by
 $$\frac{1}{X^{\frac{\beta}{2}}} \;e^{2 \pi i g_{0} \left(X,A\right)} e \left( B g_{1}\left(X,A\right) +B^{2} g_{2}\left(X,A\right) +  O \left(\frac{B^{3}}{X^{2\beta}}\right) \right),$$
 where $g_0(X,A)\ll X^{\beta}$, $g_1(X,A)\ll 1$ and $g_2(X,A)\ll 1/X^{\beta}$. It follows that the integral $\mathcal{J}$ is given by
\begin{align*}
 \frac{1}{X^{\beta}} \int W(w) \;& e \left(\left(B g_{1}\left(X,A\right)-B' g_{1}\left(X,A'\right)\right) +\left(B^{2} g_{2}\left(X,A\right)-B'^{2} g_{2}\left(X,A'\right)\right)\right) \\
\nonumber & \times  e \left( -\frac{N_0 n_2w}{q_1q_2q_2^\prime r n_1} + O\left(\frac{B^3}{ X^{2 \beta}}\right)+O\left(\frac{B^{\prime 3}}{ X^{2 \beta}}\right)\right)  \mathrm{d}w.
\end{align*}
 We note that 
$$\frac{B^3}{ X^{2 \beta}} \asymp \frac{B^{\prime 3}}{ X^{2 \beta}}\asymp \frac{XN_0}{C^3rX^{2 \beta}}.$$
Since $n_{2} \neq 0 $, we get
$$\frac{n_{2}N_{0}w}{q_1q_2q_2^{\prime}rn_1} \gg \frac{N_{0}}{(n_1,r)C^{2}r} \gg\frac{N_{0}}{C^{2}r}\gg  \frac{XN_{0}}{C^{3}rX^{2 \beta}},$$
provided $\beta >1/3$.
Making the change of variable $y \rightarrow z^3$  and applying the third derivative bound for the exponential integral, we get
$$\mathcal{J} \ll  \frac{1}{X^{\beta}} \frac{C^{2/3}r^{1/3}n_1^{1/3}}{n_2^{1/3}q_1^{1/3}N_0^{1/3}}.$$
Hence the lemma holds.
 \end{proof}
Now we estimate $S_{r}(X)$ for $C \gg X^{\frac{1}{2}-\frac{\beta}{2}+\epsilon}$. We have the following lemma. 
\begin{lemma}\label{S generic bound}
	Let $S_r(X)$ be as in \eqref{aftercauchy1}. Let $S_{r, generic}^{\neq 0}(X)$ denotes the contribution of $C \gg  X^{\frac{1}{2}-\frac{\beta}{2}}$ and $n_2 \neq 0$ to $S_r(X)$. Then we have 
	$$S_{r, generic}^{\neq 0}(X) \ll   r^{7/6}X^{3/4+\beta/12+5\eta/12}.$$
\end{lemma}
\begin{proof}
	Proof will follow using the same steps as in Lemma \ref{S(small)}. In this case, we will use $C \ll Q$ and Lemma \ref{J better} to bound $\mathcal{J}$. Using Lemma \ref{J better} and Lemma \ref{charsum} in \eqref{omega nonzero starting}, we infer that
	\begin{align*}
	\Omega_{\neq 0}&\ll\frac{N_{0}q_1r}{ X^{\beta}n_1^3}\mathop{\sum \sum}_{q_2,q_2^{\prime} \sim C/q_!} \mathop{\sum \sum}_ {\substack{\left|m\right|,\left|m'\right| \ll M_{0} \\ (m,q)=(m',q')=1 }} \sum_{\substack{\left|n_{2}\right| \ll \tilde{N} \\ n_{2} \neq 0}}(q_2,q_2^{\prime}, n_2)  \frac{C^{2/3}r^{1/3}n_1^{1/3}}{n_2^{1/3}q_1^{1/3}N_0^{1/3}}\\
	& \ll \frac{N_{0}\tilde{N}^{2/3}q_1^{2/3}r^{4/3}}{ X^{\beta}n_1^{8/3}}\frac{C^2M_0^2}{q_1^2}\frac{C^{2/3}}{N_0^{1/3}} \ll \frac{(KX)^{1/3}(Cn_1r)^{2/3}}{q_1^{2/3}}\frac{q_1^{2/3}r^{4/3}}{ X^{\beta}n_1^{8/3}}\frac{C^{8/3}M_0^2}{q_1^2} \\
	& \ll \frac{r^2}{n_1^2q_1^2}C^{16/3}K^{1/3}X^{\beta-5/3}.
	\end{align*}
	Plugging in the above bound in \eqref{aftercauchy1}, we get 
		\begin{equation}
	S_{r, generic}^{\neq 0}(X)\ll X^{\epsilon} \sup_{C \leq Q} \, \frac{X^{2/3} }{C} \sum_{\frac{n_1}{(n_1,r)} \ll C} n_1^{1/3}\Theta^{1/2}\sum_{\frac{n_1}{(n_1,r)}|q_1|(n_1r)^\infty}\frac{r}{n_1q_1}C^{8/3}K^{1/6}X^{\beta/2-5/6},
	\end{equation}
	Executing the $q_1$-sum we get
		\begin{equation}
	S_{r, generic}^{\neq 0}(X)\ll X^{\epsilon} \sup_{C \leq Q} \, \frac{X^{2/3} }{C} \sum_{\frac{n_1}{(n_1,r)} \ll C} \frac{(n_1,r)}{n_1^{5/3}}\Theta^{1/2}C^{8/3}K^{1/6}X^{\beta/2-5/6},
	\end{equation}
	$n_1$-sum can be evaluated as follows:
\begin{align*}
\sum_{n_1\ll Cr}\frac{(n_1,r)}{n_1^{5/3}}\Theta^{1/2} \ll \sum_{n_1\ll Cr}\frac{(n_1,r)^{1/2}}{n_1^{1/2}}\frac{(n_1,r)^{1/2}}{n_1^{7/6}}\Theta^{1/2} \ll N_0^{1/6}.
\end{align*}
Hence we have 
\begin{align*}
	S_{r, generic}^{\neq 0}(X) \ll  \sup_{C \leq Q} \, \frac{X^{2/3} }{C}N_0^{1/6}C^{8/3}K^{1/6}X^{\beta/2-5/6} \ll r^{7/6}X^{3/4+\beta/12+5\eta/12}.
\end{align*}
\end{proof}

 \section{\bf Conclusion}
In this  section, we bring together all the estimates from Lemma \ref{S(zero)}, Lemma \ref{S(small)} and Lemma \ref{S generic bound}. Hence we get
\begin{align*}
	S_r(X)&\ll \sqrt{\alpha \beta}r^{7/6}X^{\frac{3 \beta}{4}+\frac{3}{4}-\frac{3 \eta}{4}}+ \alpha \beta r^{7/6}X^{\frac{3}{4}+ \frac{\beta}{4}-\frac{\eta}{2}}+ r^{7/6}X^{3/4+\beta/12+5\eta/12}. \\
	&\ll \alpha\sqrt{ \beta}r^{7/6}X^{\frac{3 \beta}{4}+\frac{3}{4}-\frac{3 \eta}{4}}+ \alpha \sqrt{\beta} r^{7/6}X^{\frac{3}{4}+ \frac{\beta}{4}-\frac{\eta}{2}}+\alpha \sqrt{\beta} r^{7/6} r^{7/6}X^{3/4+\beta/12+5\eta/12}.
\end{align*}
As $\beta > \eta $, it follows that first term dominates the second term. Hence we have 

\begin{align*}
	S_r(X)\ll  \alpha\sqrt{ \beta}r^{7/6}X^{\frac{3 \beta}{4}+\frac{3}{4}-\frac{3 \eta}{4}}+\alpha \sqrt{\beta} r^{7/6}X^{3/4+\beta/12+5\eta/12}.
\end{align*}
Upon equating the above terms, we see that the optimal choice of $\eta$ is given by 
$$\eta=4\beta/7.$$
With this choice we get 
$$S_r(X) \ll \alpha \sqrt{\beta}r^{7/6}X^{3/4+9\beta/28}.$$
Note that we have a non-trivial power saving bound if $1/3<\beta < 7/9$. Thus we get Theorem \ref{mainthe}.
%

\section{\bf Proof of Theorem \ref{gl3thm21} }
In this section we will prove Theorem \ref{gl3thm21}. We first recall its statement.
\paragraph{\bf Theorem 2}
  Let $\lambda_f (n)$ be the Fourier coefficients of a Maass form $f$ for $SL(2, \mathbb{Z}) $. Let $\lambda_\pi(r, n)$ be the Fourier-Whittaker coefficients of a Maass form $\pi$ for
$SL(3, \mathbb{Z})$.  Then for any real $\alpha$, $0 < \beta <1$ and integer $ r \geq 1$ we have
\begin{equation}\label{R(X)}
R(X) := \sum_{n=1}^{\infty} \lambda_{\pi}(r,n) \, \lambda_{f}(n) \, e\left(\alpha n^{\beta}\right) V\left(\frac{n}{X}\right) \ll_{\pi, f,\epsilon}(\alpha \beta)^{\frac{3}{2}} r^{\frac{7}{6}}X^{\frac{3}{4}+\frac{9\beta}{14}+\epsilon},
\end{equation}
where $V(x)$ is a smooth function supported in $[1,2]$  satisfying  $V^{(j)}(x) \ll_{j} 1$ and $\int V=1$.
\begin{proof}
	
As mentioned in the Introduction, proof of Theorem \ref{gl3thm21} follows by applying the same machinery as in Theorem \ref{Thm} in this setup. We will see that, this time, we will use $GL(2)$ Voronoi summation formula instead of the Poisson summation formula to dualize $m$-sum  and accordingly there will be corresponing changes in the character sum and the integral transform. Now we proceed towards the proof. 
\subsection{Application of delta method} As a first step, we separate oscillations involved in $R(X)$ in \eqref{R(X)}. Indeed, there are three oscillatory factors, namely, $ \lambda_{\pi}(r,n)$,  $ \lambda_{f}(n)$ and $e\left(\alpha n^{\beta}\right)$ in $R(X)$. We will separate them using delta method.  Also we introduce a $v$-integral to lower the conductor. Thus $R(X)$ can be rewritten  as
\begin{align*}
 R(X) = \frac{1}{K} \int_{\mathbb{R}} V\left(\frac{v}{K}\right)\mathop{\sum \sum}_{\substack{n, m=1 \\ n = m } }^{\infty} \lambda_{\pi}(r,n)\lambda_{f}(n) e\left(\alpha m^{\beta}\right) \left(\frac{n}{m}\right)^{iv} V\left(\frac{n}{X}\right) U\left(\frac{m}{X}\right)\mathrm{d}v,
\end{align*} 
 where $K=X^{\beta -\eta} < X^{\beta}$ is a parameter which will be chosen later optimally, and $U(x)$ is a smooth function supported in $[1/2,5/2]$, with $U(x)=1$ for $x \in [1,2]$, and $U^{(j)}(x) \ll_{j}1$.
 As mentioned in Section \ref{aplicircle}, $v$-integral gives us restrictions on $|n-m|$,
 $$|n-m| \ll \frac{X}{K} X^{\epsilon}.$$ 
On applying Lemma \ref{deltasymbol} with $Q = X^{\epsilon}\sqrt{X/K}$ to $R(X)$, we get
 \begin{align}\label{R(X) after delta}
  R(X) =&\frac{1}{QK}\int_{\mathbb{R}}W(x)\int_{\mathbb R}V\left(\frac{v}{K}\right)\sum_{1\leq q\leq Q}\;\frac{g(q,x)}{q}\;\sideset{}{^\star}\sum_{a\bmod{q}} \notag \\
  &\times \sum_{m=1}^\infty \lambda_f(m)m^{-iv} e\left(\alpha m^{\beta} \right)e\left(-\frac{am}{q}\right)e\left(-\frac{mx}{qQ}\right)U\left(\frac{m}{X}\right)\notag \\
 &\times \mathop{\sum}_{n=1}^\infty \lambda_\pi(r,n)e\left(\frac{an}{q}\right)e\left(\frac{nx}{qQ}\right) n^{iv} V\left(\frac{n}{X}\right) \mathrm{d}v\mathrm{d}x  + O(X^{-2020}).
 \end{align}
 \subsection{GL(2) Voronoi formula} Next, we apply $GL(2)$ Voronoi summation formula to the $m$-sum in $R(X)$ in \eqref{R(X) after delta}. On Applying Lemma \ref{gl2 voronoi} to the $m$-sum we obtain
\begin{align*}
&\sum_{m =1}^\infty \, \lambda_{f}(m) \, m^{-iv} e\left(\alpha m^{\beta}\right) e\left(\frac{-am}{q}\right) \,  e\left(\frac{-mx}{qQ}\right)  U\left( \frac{m}{X}\right) = \frac{2 \pi i^k}{q} \sum_{m =1}^\infty \, \lambda_{f}(m) \\
&   \hspace{1cm} \times e\left(\frac{m \overline{a}}{q}\right) \int_0^\infty U(y/X) y^{-iv}  e \left( \alpha  y^{\beta} \right) \left( \frac{- x y}{qQ}\right) J_{k-1} \left( \frac{4 \pi \sqrt{m y}}{q}\right) dy. 
\end{align*}
Using the change of variable, $y \rightsquigarrow yX$, and using  Lemma \ref{bessel function decompo} for the bessel function, we see that the $m$-sum is given by
\begin{align} \label{I1}
\frac{X^{3/4-iv}}{\sqrt{q}} \sum_{m =1}^\infty  \frac{\lambda_{f}(m)}{m^{1/4}}  e\left(\frac{m \overline{a}}{q}\right) \int_0^\infty U(y) y^{-iv} e\left(  \alpha X^{\beta} y^{\beta} - \frac{ X x y}{qQ}  \pm \frac{2 \sqrt{m Xy}}{q} \right) dy. 
\end{align} 
Note that we have made a slight abuse of notation, the weight function $U$ above is different from the one we started with. Let us denote the integral in above equation  by $I_1(m, q, x)$. By repeated integration by parts $j$-times we see that 
\begin{align*}
I_1(m, q, x) \ll \left( 1+ K + \alpha \beta X^\beta + \frac{Xx}{q Q} \right)^j \left( \frac{q}{ \sqrt{mX}}\right)^j \ll  \left( \alpha \beta X^\beta + \frac{X}{q Q} \right)^j \left( \frac{q}{ \sqrt{mX}}\right)^j. 
\end{align*}
We observe that integral $ I_1(m, q, x)$ is negligibly small unless
\begin{align} \label{second M 0}
m \ll X^\epsilon  \ \max \left(\alpha^2\beta^2 q^2 X^{2\beta-1}, \  K \right) := M_0. 
\end{align}
We end this subsection by recording the above discussion in the following lemma.
\begin{lemma}\label{dual m sum}
	Let $g(y)= y^{-iv} e\left(\alpha y^{\beta}\right)  e\left({-yx}/{qQ}\right)  U\left(y/X \right)$. Let $M_0$ be as in \eqref{second M 0}. Then we have
	\begin{align*}
		\sum_{m \ll M_0} \, \lambda_{f}(m)e\left(\frac{-am}{q}\right)g(m)=\frac{X^{3/4-iv}}{\sqrt{q}} \sum_{m =1}^\infty  \frac{\lambda_{f}(m)}{m^{1/4}}  e\left(\frac{m \overline{a}}{q}\right) I_1(m, q, x),
	\end{align*}
	where 
	\begin{align*}
	I_1(m, q, x)=\int_0^\infty U(y) y^{-iv} e\left(  \alpha X^{\beta} y^{\beta} - \frac{ X x y}{qQ}  \pm \frac{2 \sqrt{m Xy}}{q} \right) dy. 
	\end{align*}
\end{lemma}
\subsection{GL(3) Voronoi} We now apply $GL(3)$ Voronoi summation formula to the $n$-sum in \eqref{R(X) after delta}. We note that the same $n$-sum appeared in Subsection \ref{GL3 voro}. For the sake of completeness, we repeat  Lemma \ref{first gl3 voronoi} of Subsection \ref{GL3 voro}. 
\begin{lemma}\label{2nd gl3 voronoi}
	Let $\psi(n)= n^{iv} \, e\left(\frac{nx}{qQ}\right)  V\left(\frac{n}{X}\right)$ and $N_0$ be as in \eqref{n-sumbound}. Then we have 
	\begin{align*}
	\mathop{\sum}_{n=1}^\infty \lambda_\pi(r,n)e\left(\frac{an}{q}\right)\psi(n)=&\frac{X^{2/3+iv}}{qr^{2/3}} \sum_{n_1|qr}n_1^{1/3}\sum_{n_2\ll N_0/n_1^2}\frac{\lambda_\pi(n_1,n_2)}{n_2^{1/3}}S(r\bar a,  n_2; qr/n_1)\\
	\nonumber &\times \int_0^\infty V(z)z^{iv}e\left(\frac{Xxz}{qQ}\pm \frac{3(Xn_1^2n_2z)^{1/3}}{qr^{1/3}}\right)\mathrm{d}z.
	\end{align*}
\end{lemma}

and $N_0$ is given by \eqref{n-sumbound}.  We note that there are $4$ integral in above expression. 

\subsection{Simplification of the integrals} 
On substituting Lemma \ref{dual m sum} and Lemma \ref{2nd gl3 voronoi} in \eqref{R(X) after delta} we get that 
\begin{align}\label{R(X) after summation formula}
R(X)=& \frac{X^{17/12}}{QKr^{2/3}} \int_{\mathbb{R}}W(x)\int_{\mathbb R}V\left(\frac{v}{K}\right)\sum_{ q\leq Q}\frac{g(q,x)}{q^{5/2}}   \sideset{}{^\star}{\sum}_{a \, \rm mod \, q} \notag\\
& \times \sum_{m \ll M_0}  \frac{\lambda_{f}(m)}{m^{1/4}} e\left(\frac{m \overline{a}}{q}\right) I_1(m, q, x) \sum_{n_1|q}n_1^{1/3}  \notag \\
&  \times   \sum_{n_2 \ll {N_{0}}/ {n_{1}^2}} \frac{\lambda_\pi(n_1,n_2)}{n_2^{1/3}}S(r\bar a,  n_2; qr/n_1) I_2(n, q, x) \mathrm{d}v \mathrm{d}x, 
\end{align}
 where 
\begin{equation}
I_2(n, q, x) := \int_0^\infty V(z)z^{iv}e\left(\frac{Xxz}{qQ}\pm \frac{3(Xn_1^2n_2z)^{1/3}}{qr^{1/3}}\right)\mathrm{d}z. 
\end{equation} 
 We now consider the above four fold integral 

\begin{align*}
 &\int_{\mathbb{R}}W(x)\int_{\mathbb R}V\left(\frac{v}{K}\right) g(q,x) \int_0^\infty U(y)y^{-iv}e\left(\alpha X^{\beta}y^{\beta} \right)\int_0^\infty   V(z)z^{iv} \\
& \times   e\left( \frac{Xx(z-y)}{qQ}   \pm  \frac{2 \sqrt{m Xy}}{q} \pm \frac{3(Xn_1^2n_2z)^{1/3}}{qr^{1/3}}\right) \mathrm{d}z\mathrm{d}y \,  \mathrm{d}v \, \mathrm{d}x.  
\end{align*}
We observe that this is similar as the four integral in Section \ref{4 integral}. Thus like in Subsection \ref{4 integral} considering  the $x$ integral we obtain  the restriction $|z-y| \ll X^\epsilon q/ QK$.   Writing $z= y+ u$ with $|u|\ll X^\epsilon q/Q K$, we see that the $y$-integral changes to
\begin{align} \label{second I}
I(m, n_1^2 n_2, q)& :=  \int_{\mathbb R}   U(y) 
e\left( \alpha X^{\beta} y^{\beta}  \pm  \frac{2 \sqrt{m Xy}}{q} \pm \frac{3(Xn_1^2n_2(y+u))^{1/3}}{qr^{1/3}}\right) \   \mathrm{d}y.
\end{align}
Thus, using the same arguments as in Subsection \ref{4 integral}, the four fold integral changes into
$$K \times \frac{q}{QK}\times  I \left(m,n_{1}^{2}n_{2},q \right).$$


\subsection{Cauchy-Schwarz inequality}
After simplification of the integrals, $R(X)$ in  \eqref{R(X) after summation formula} has essentially reduced to

\begin{align} \label{RX before chauchy}
& \frac{X^{17/12}}{Q^2Kr^{2/3}}\sum_{1\leq q\leq Q}\frac{1}{q^{3/2}}  \sideset{}{^\star}{\sum}_{a \, \rm mod \, q} \notag \\
& \times \sum_{n_1|q}n_1^{1/3} \sum_{n_2 \ll \frac{N_{0}} {n_{1}^2}} \frac{\lambda_\pi(n_1,n_2)}{n_2^{1/3}} S(r\bar a,  n_2; qr/n_1) \notag\\
&\times \sum_{m \ll M_0} \frac{\lambda_{f}(m)}{m^{1/4}}e\left(\frac{m \overline{a}}{q}\right)I(m, n_1^2 n_2, q).
   \end{align} 
Spliting the sum over $q$ in dyadic blocks $q\backsim C$ and  writing $q=q_1q_2$ with $q_1|(n_1r)^\infty$, $(n_1r,q_2)=1$, we see that $R(X)$  is dominated by

\begin{align*}
R(X)\ll  \sup_{C\leq Q} \frac{X^{5/12} }{C^{3/2}r^{2/3}} & \sum_{\frac{n_1}{(n_1,r)}\ll C}n_1^{1/3}\sum_{\frac{n_1}{(n_1,r)}|q_1|(n_1r)^\infty}\sum_{n_2\ll {N_0}/{n_1^2}} \frac{|\lambda_\pi(n_1,n_2)|}{n_2^{1/3}} \notag \\
&  \times \Big|\sum_{ q_2 \sim C/{q_1}}\sum_{m \ll M_0}  \frac{\lambda_{f}(m) }{m^{1/4}}\mathcal{C}(m, n_1,n_2,q)  I(m, n_1^2 n_2; q)\Big|,
\end{align*} 
where the character sum $\mathcal{C}(m, n_1,n_2,q) $  is defined as 
\begin{align*}
\mathcal{C}(...):=\sideset{}{^\star}\sum_{a \, {\rm mod} \, q}S(r \bar a,  n_2; qr/n_1)e\left(\frac{\bar{a}m}{q}\right)=\sum_{d|q}d\mu\left(\frac{q}{d}\right)\sideset{}{^ \star}\sum_{\substack{\alpha \, {\rm mod} \, qr/n_1 \\ n_1\alpha\equiv-m \, {\rm mod} \, d}}e\left(\frac{\bar{\alpha}n_2}{qr/n_1}\right).
\end{align*}
We  now analyze the sum inside $| \ |$. We split the $m$-sum into dyadic blocks $m \sim M_1$. On  applying the Cauchy's inequality to the $n_2$-sum ,  we get the following bound for $R(X)$:
\begin{align}\label{R(X) after cauchy}
R(X) \ll \mathop{\sup_{ \substack{M_1\ll M_0 \\ C \ll Q}}}\frac{X^{5/12} }{C^{3/2}r^{2/3}}\sum_{\frac{n_1}{(n_1,r)}\ll C}n_1^{1/3}\Theta^{1/2}\sum_{\frac{n_1}{(n_1,r)}|q_1|(n_1r)^\infty}\sqrt{\Omega} \ ,
\end{align}  
where 
\begin{align}
\Theta=\sum_{n_2\ll N_0/n_1^2} \frac{|\lambda_\pi(n_1,n_2)|^2}{n_2^{2/3}},
\end{align} 
and 
\begin{align}\label{omega 1}
\Omega = \sum_{n_2 \ll N_0/ n_1^2}  \Big|   \sum_{ q_2 \sim C/{q_1}}   \sum_{m \sim M_1} \frac{\lambda_{f}(m) }{m^{1/4}}   \mathcal{C} (...)  I(m, n_1^2 n_2, q)\Big|^2.
\end{align}
\subsection{Poisson summation formula}
In this subsection, we will analyze $\Omega$ in \eqref{omega 1} . Analysis of $\Omega$  is  similar to the one which was carried  out in Subsection \ref{Poisson}. Thus, proceeding as in Subsection \ref{Poisson}, opening the absolute value square and applying the Poisson summation formula over $n_2$ with modulus $q_1 q_2 q_2^\prime r/n_1$ we arrive at
\begin{align}\label{omega 2}
\Omega \ll \frac{N_0}{n_1^2 M_1^{1/2}} \mathop{ \sum  \sum}_{ q_2 \,  q_2^\prime  \sim C/{q_1}}  \mathop{ \sum \sum}_{m , \,  m^\prime\sim M_1} \sum_{n_2 \in \mathbb{Z}}  \left| \mathfrak{C}\right| \left| \mathfrak{J}\right|, 
\end{align}
where 
\begin{align}\label{c}
\mathfrak{C}=\mathop{\sum \sum}_{\substack{d|q \\ d^{\prime}|q^\prime}}dd^\prime\mu\left(\frac{q}{d}\right)\mu\left(\frac{q^\prime}{d^\prime}\right)\mathop{\sideset{}{^ \star}\sum_{\substack{\alpha \, {\rm mod} \, qr/n_1 \\ n_1\alpha\equiv-m \, {\rm mod} \, d}} \  \sideset{}{^ \star}\sum_{\substack{\alpha^\prime \, {\rm mod} \, q^\prime r/n_1 \\ n_1\alpha^\prime \equiv-m^\prime \, {\rm mod} \, d^\prime}}}_{\bar{\alpha}q_2^\prime -\bar{\alpha}^\prime q_2\equiv -n_2  \, {\rm mod} \, q_1q_2q_2^\prime r/n_1}1
\end{align}
 and 
\begin{align} \label{last integral}
\mathfrak{J} = \int_{\mathbb{R}} W(w)I(m,N_0w, q) \overline{I(m^\prime,N_0 w, q^\prime)}  e \left( - \frac{N_0  n_2w}{q_1  q_2 q_2^\prime rn_1}\right) \mathrm{d} w. 
\end{align}
Note that we have used Deligne's bound \eqref{gl2 ramanujan} to estimate $\lambda_{f}(m)$ and $\lambda_{f}(m^{\prime})$. By repeated integration by parts we see that the integral $\mathfrak{J}$ is negligibly small unless 
\begin{align}\label{n_2 star}
n_{2} \ll  \frac{\sqrt{XK} Cn_1r}{N_{0}q_1} X^{\epsilon}:= \tilde{N}.
\end{align}
Note that this is the same $\tilde{N}$ as in Subsection \ref{Poisson}.
\subsection{Character sum analysis} In this subsection, we will analyze $\mathfrak{C}$ in \eqref{c}. We note that same character sum appeared in \cite{munshi2}. The following lemma is taken from \cite{munshi2}. 
\begin{lemma} \label{bound for c}
	Let  $\mathfrak{C}$ be as in \eqref{c}. Then, for $n_2=0$, we have $q_2=q_2^{\prime}$ and
	\begin{align*}
	\mathfrak{C}  \ll  \mathop{\sum \sum}_{\substack{d ,d^{\prime} \vert q \\ (d,d^{\prime})|(m-m^{\prime})}} dd' \frac{qr}{[d,d^{\prime}]}. 
	\end{align*}
	For $n_2 \neq 0$, we have 
	\begin{align*}
	\mathfrak{C} \ll \frac{q_{1}^2 \, r (m,n_{1})}{n_{1}} \mathop{\sum \sum}_{\substack{d_{2} \mid (q_{2},  n_{1} q_{2}^{\prime}- mn_{2}) \\ d_{2}^{\prime} \mid (q_{2}^{\prime},  n_{1} q_{2} + m^{\prime} n_{2})}} \, d_{2} d_{2}^{\prime} \, .
	\end{align*}
\end{lemma}
\begin{proof}
	In the case $n_2=0$, it follows from  the congruence conditions in the definition  of  $\mathfrak{C}$ in \eqref{c} that   $ \bar{\alpha}q_2^\prime -\bar{\alpha}^\prime q_2\equiv 0  \, {\rm mod} \, q_1q_2q_2^\prime r/n_1$, which implies that $q_2=q_2^{\prime}$ and $\alpha=\alpha^{\prime}$. So we can bound the character sum  $\mathfrak{C}$   as 
	\begin{align*}
	\mathfrak{C}  \ll   \mathop{\sum \sum}_{\substack{d ,d^{\prime} \vert q}} dd'  \mathop{\sideset{}{^\star} \sum_{\alpha \; {\rm mod} \; qr/n_1 } }_{\substack{n_1\alpha\equiv-m \, {\rm mod} \, d \\ n_1\alpha \equiv-m^\prime \, {\rm mod} \, d^\prime  }} 1 \ll \mathop{\sum \sum}_{\substack{d ,d^{\prime} \vert q \\ (d,d^{\prime})|(m-m^{\prime})}} dd' \frac{qr}{[d,d^{\prime}]}.
	\end{align*}
	Hence we get the first part of the lemma. For the second part, using the Chinese Remainder theorem,  we observe that $\mathfrak{C}$ can be dominated by a product of two sums $ \mathfrak{C} \ll \mathfrak{C}^{(1)} \mathfrak{C}^{(2)}$,
	where
	$$\mathfrak{C}^{(1)} = \mathop{\sum \sum}_{\substack{d_{1}, d_{1}^{\prime} | q_{1}}} d_{1} d_{1}^{\prime}  \; \mathop{\sideset{}{^\star}{\sum}_{\substack{\beta \; \rm mod \; \frac{q_{1}r}{n_{1}} \\  n_{1} \beta \;  \equiv \; - m \; \rm mod \;  d_{1}}} \  \sideset{}{^\star}{\sum}_{\substack{\beta^{\prime} \; \rm mod \; \frac{q_{1}r}{n_{1}} \\  n_{1}   \beta^{\prime} \;  \equiv \; - m^{\prime} \; \rm mod \;  d_{1}^{\prime}}}}_{ \overline{\beta} q_{2}^{\prime} - \overline{\beta^{\prime}} q_{2} + n_{2} \; \equiv \; 0 \; {q_{1}r }/{n_{1}} } \; 1 $$
	and
	$$\mathfrak{C}^{(2)} = \mathop{\sum \sum}_{\substack{d_{2} \mid q_{2} \\ d_{2}^{\prime} \mid q_{2}^{\prime}}} d_{2} d_{2}^{\prime}  \; \mathop{\sideset{}{^\star}{\sum}_{\substack{\beta \; \rm mod \; q_{2} \\  n_{1} \beta \;  \equiv \; - m \; \rm mod \;  d_{2}}} \, \sideset{}{^\star}{\sum}_{\substack{\beta^{\prime} \; \rm mod \; q_{2}^{\prime} \\  n_{1}   \beta^{\prime} \;  \equiv \; - m^{\prime}  \; \rm mod \;  d_{2}^{\prime}}}}_{ \overline{\beta} q_{2}^{\prime} - \overline{\beta^{\prime}} q_{2} + n_{2} \; \equiv \; 0 \; q_{2} q_{2}^{\prime} } \; 1.$$
	
	In the second sum $\mathfrak{C}_{\pm}^{(2)}$, since  $(n_{1},q_{2}q_{2}^{\prime})=1$,  we get  $\beta \equiv -m\bar{n_1}  \rm \, mod \, d_{2}$ and $\beta^{\prime} \equiv \, -m^{\prime}\bar{n_1} \, \rm \, mod \, d_{2}^{\prime}$. Then using the congruence modulo $q_{2} q_{2}^{\prime}$, we conclude that
	
	$$\mathfrak{C}^{(2)} \ll  \mathop{\sum \sum}_{\substack{d_{2} \mid (q_{2},  n_{1} q_{2}^{\prime}- mn_{2}) \\ d_{2}^{\prime} \mid (q_{2}^{\prime},  n_{1} q_{2} + m^{\prime} n_{2})}} \, d_{2} d_{2}^{\prime}.$$
	In the first sum $\mathfrak{C}_{\pm}^{(1)}$, the congruence condition determines $\beta^{\prime}$ uniquely in terms of $\beta$, and hence 
	$$\mathfrak{C}^{(1)} \ll \mathop{\sum \sum}_{\substack{d_{1}, d_{1}^{\prime} | q_{1}}} d_{1} d_{1}^{\prime} \sideset{}{^\star}{\sum}_{\substack{\beta \; \rm mod \; {q_{1}r}/{n_{1}} \\  n_{1} \beta \;  \equiv \; - m \; \rm mod \;  d_{1}}} \, 1 \ll \frac{r \, q_{1}^2 \, (m,n_{1}) }{n_{1}} .$$
	Hence we have the lemma.
\end{proof}

\subsection{Stationary Phase analysis}
In this subsection, we will analyze the integral transform $\mathfrak{J}$ given in \eqref{last integral}. We have the following lemma.
\begin{lemma}\label{second L^2 bound}
	Let $\mathfrak{J}$ be as in \eqref{last integral}.  Then we have 
	\begin{align*}
		\mathfrak{J}\ll 1/X^{\beta}.
	\end{align*}
	Moreover, for $C \gg X^{1/2-\beta/2}$, we have
		$$\mathcal{J} \ll   \frac{1}{X^{\beta}} \frac{C^{2/3}r^{1/3}n_1^{1/3}}{n_2^{1/3}q_1^{1/3}N_0^{1/3}}.$$
\end{lemma}
\begin{proof}
We mention here that the proof of the lemma will be similar to that of Lemma \ref{J better}. To start with, we first analyze $I(m,N_0w, q)$ which is given as (see \eqref{second I})
	\begin{align} 
	I(m, N_0 w, q)=  \int_{\mathbb R}   U(y) 
	e\left( \alpha X^{\beta} y^{\beta}  \pm  \frac{2 \sqrt{m Xy}}{q} \pm \frac{3(XN_0w(y+u))^{1/3}}{qr^{1/3}}\right) \   \mathrm{d}y.
	\end{align}
	Let
	$$A= \frac{2 \sqrt{m X}}{q} \quad \text{and} \;  B= \frac{3 \left(X N_{0}w\right)^{1/3}}{qr^{1/3}}.$$
	As we observed in the proof of Lemma \ref{J better} that we may ignore $u$ in the phase function of $I(m,N_0w,q)$. 
	Thus the phase function in the exponential integral $I(...)$ is essentially given by
	\begin{align}\label{phase}
		f(y)= \alpha X^{\beta} y^{\beta} \pm A \sqrt{y} \pm B y^{1/3}.
	\end{align}
	Recall that $$A=\frac{2 \sqrt{m X}}{q} \ll \frac{2 \sqrt{M_0 X}}{C}=\max\left(\alpha\beta X^{\beta}, \frac{\sqrt{KX}}{C}\right).$$
	Thus
	 \begin{align}\label{M0 size for large C}
	 	\frac{2 \sqrt{M_0 X}}{C} \asymp X^{\beta} \iff C \gg X^{1/2-\beta/2-\eta/2} \iff  \frac{3 \left(X N_{0}w\right)^{1/3}}{qr^{1/3}} \ll X^{\beta}.
	 \end{align}
	In other words, if $C \ll  X^{1/2-\beta/2-\eta/2}$, then
	$$ \frac{2 \sqrt{M_0 X}}{C} \asymp \frac{3 \left(X N_{0}w\right)^{1/3}}{Cr^{1/3}} \asymp \frac{\sqrt{KX}}{C}\gg X^{\beta}.$$
	Hence if $C \ll  X^{1/2-\beta/2-\eta/2-\epsilon}$, then the $y$-integral will be negligibly small unless 
	$$A \asymp \frac{2 \sqrt{M_0 X}}{C} \iff M_1 \asymp M_0.$$
	 Also if $C \gg  X^{1/2-\beta/2-\eta/2+\epsilon}$, then the $y$-integral is negligibly small unless 
	$$A \asymp \frac{2 \sqrt{M_0 X}}{C} \iff M_1 \asymp M_0.$$ 
	 We make a change of varible $y \rightsquigarrow y^2$ so that the new phase function looks like
		$$f(y)= \alpha X^{\beta} y^{2\beta} \pm A y \pm B y^{2/3}.$$
		On computing the second order derivative we get
			$$f^{\prime \prime}(y)= 2\alpha \beta (2\beta -1) X^{\beta} y^{2\beta-2} \mp \frac{2B}{9y^{4/3}}.$$
		If $B \asymp X^{\beta}$ and there is a negative sign in the second term, then using the same arguments as in Lemma \ref{L^2bound},  we get 
			\begin{align*}
				\mathfrak{J}\ll \int_{\mathbb{R}} W(w) \vert \mathcal{I}(m,N_{0}w,q)\vert^{2} \mathrm{d}w \ll \frac{1}{B} \asymp \frac{1}{X^{\beta}}.  
			\end{align*} 
			In the other situations, using the second derivative bound, we get
			$$I(m, N_0 w, q) \ll1/X^{\beta/2}.$$
			And hence first part of the lemma follows. To prove the second part, we  go back to the phase function in \eqref{phase}
				$$f(y)= \alpha X^{\beta} y^{\beta} \pm A \sqrt{y} \pm B y^{1/3}.$$
			 Note that for $q \sim C \gg X^{{1/2-\beta/2}+ \epsilon}$, we have $B \ll X^{\beta - \eta/2}$ and  $X^{\beta} \gg A $. Also if the second term has positive sign or  $A \ll X^{\beta-\epsilon}$ then the $y$-integral is negligibly small. Hence, we may  assume that the second term has negative sign and $A \asymp X^{\beta}$. Next, we  find the stationary point, say, $y_{\text{stat}}$ of the above phase function. In fact, it can be written as
	$$y_{\text{stat}} = y_{0}+y_{1}+y_2+ \ldots,$$
	where $y_{k}= h_{k}(X,A) \left(B/X^{\beta}\right)^{k}$ with $h_{k}(X,A) \ll 1$ for  $k \geq 0$.   Explicit calculation yields
	$$y_{0}= \left(\frac{2\alpha \beta X^{\beta}}{A}\right)^{\frac{2}{1- 2\beta}}, \quad y_{1}= \frac{\mp B}{3 \alpha \beta \left(\beta -1\right) X^{\beta}} \left(\frac{2\alpha \beta X^{\beta}}{A}\right)^{\frac{2(4-3 \beta)}{3 \left(1-2\beta\right)}} .$$
	Therefore, using the stationary phase analysis Lemma \ref{stationaryphase}, the integral $I(...)$ is essentially given by
	$$\frac{1}{X^{\frac{\beta}{2}}} \;e^{2 \pi i g_{0} \left(X,A\right)} e \left( B g_{1}\left(X,A\right) +B^{2} g_{2}\left(X,A\right) +  O \left(\frac{B^{3}}{X^{2\beta}}\right) \right),$$
	where $g_0(X,A)\ll X^{\beta}$, $g_1(X,A)\ll 1$ and $g_2(X,A)\ll 1/X^{\beta}$. Similar analysis can be done for $I(m^\prime,N_0 w, q^\prime)$. Hence  the integral $\mathcal{J}$ is given by
	\begin{align*}
	\frac{1}{X^{\beta}} \int W(w) \;& e \left(\left(B g_{1}\left(X,A\right)-B' g_{1}\left(X,A'\right)\right) +\left(B^{2} g_{2}\left(X,A\right)-B'^{2} g_{2}\left(X,A'\right)\right)\right) \\
	\nonumber & \times  e \left( -\frac{N_0 n_2w}{q_1q_2q_2^\prime r n_1} + O\left(\frac{B^3}{ X^{2 \beta}}\right)+O\left(\frac{B^{\prime 3}}{ X^{2 \beta}}\right)\right)  \mathrm{d}w.
	\end{align*}
	We note that 
	$$\frac{B^3}{ X^{2 \beta}} \asymp \frac{B^{\prime 3}}{ X^{2 \beta}}\asymp \frac{XN_0}{C^3rX^{2 \beta}}.$$
	Since $n_{2} \neq 0 $, we get
	$$\frac{n_{2}N_{0}w}{q_1q_2q_2^{\prime}rn_1} \gg \frac{N_{0}}{(n_1,r)C^{2}r} \gg\frac{N_{0}}{C^{2}r^2}\gg  \frac{XN_{0}}{C^{3}rX^{2 \beta}},$$
	provided $\beta \geq 1/3$.
		Making the change of variable $y \rightarrow z^3$  and applying the third derivative bound for the exponential integral, we get
	$$\mathcal{J} \ll   \frac{1}{X^{\beta}} \frac{C^{2/3}r^{1/3}n_1^{1/3}}{n_2^{1/3}q_1^{1/3}N_0^{1/3}}.$$
	Hence the lemma holds.
\end{proof}
\begin{lemma}\label{m restriction}
		Let $\mathfrak{J}$ be as in \eqref{last integral}. Then for $n_2=0$, $\mathfrak{J}$ is negligibly small unless
	$$m- m^{\prime} \ll \frac{CM_1^{1/2}}{X^{1/2}}.$$
\end{lemma}
\begin{proof}
	First note that for $n_2=0$, by Lemma \ref{bound for c}, we have $q_2=q_2^{\prime}$. Thus on substituting $n_2=0$, the expressions for $I(m, N_0 w, q)$ and $I(m^{\prime}, N_0 w, q^{\prime})$ in $\mathfrak{J}$, we get
	\begin{align*}
		\mathfrak{J}=& \int \int U(y_{1}) U(y_{2})e\left( \alpha X^{\beta} (y_1^{\beta}-y_2^{\beta})  \pm  \frac{2 \sqrt{m Xy_1}}{q} \mp \frac{2 \sqrt{m^{\prime} Xy_2}}{q} \right) \\
		& \times \int W(w) e \left(\frac{3(XN_{0}w)^{1/3}}{qr^{1/3}}\left(\pm(y_{1}+u)^{1/3}\mp(y_{2}+u)^{1/3}\right)\right) \, \mathrm{d}w \ \mathrm{d}y_1 \ \mathrm{d}y_2
	\end{align*}
	Making a change of variable $w \rightsquigarrow w^3 $ in the $w$-integral, we see that it is negligibly small unless $$|y_1-y_2| \ll Cr^{1/3}/(XN_0)^{1/3} \ll 1/K.$$
	Next taking $y_1=y_2+u_1$, with $u_1 \ll 1/K$, and considering the $y_2$-integral, we see that it is negligibly small unless $$m- m^{\prime} \ll \frac{CM_1^{1/2}}{X^{1/2}}.$$
	Hence the lemma follows.
	\end{proof}

\subsection{Zero frequency} In this subsection, we  estimate $R(X)$ in \eqref{R(X) after cauchy} when $n_2=0$. Estimates in this subsection are similar to those in Section \ref{zero freq}. Let $\Omega_{0}$  denotes the part  of $\Omega$ in \eqref{omega 1} corresponding to $n_2=0$ and let $R^{0}(X)$ denotes the part of $R(X)$ in \eqref{R(X) after cauchy} corresponing to $\Omega_{0}$. We have  the following lemma. 
\begin{lemma} \label{R(X) zero frq bound}
	Let $\Omega_0$ and $R^{0}(X)$ be as above. Then we have
		\begin{align*}
\Omega_0\ll  \frac{N_0 \left| \mathfrak{J}\right|M_1^{1/2} C^3r}{n_1^2 q_1 },
	\end{align*}
	and 
 \begin{align*}
	R^{0}(X)  \ll r^{1/2}(\alpha\beta)^{1/2}X^{3/4+3\beta/4-3\eta/4}.
	 \end{align*} 
	\end{lemma}
\begin{proof}
	Using the bound \eqref{omega 2} for $\Omega$, we see that
	\begin{align*}
	\Omega_0 \ll \frac{N_0}{n_1^2 M_1^{1/2}} \mathop{ \sum  \sum}_{ q_2 \,  q_2^\prime  \sim C/{q_1}}  \mathop{ \sum \sum}_{m , \,  m^\prime\sim M_1}  \left| \mathfrak{C}\right| \left| \mathfrak{J}\right|. 
	\end{align*}
	Now using   Lemma \ref{bound for c} and Lemma \ref{m restriction}, we arrive at
	 \begin{align*}
	\Omega_0 & \ll \frac{N_0|\mathfrak{J}|}{n_1^2 M_1^{1/2} }\sum_{q_2 \sim C/q_1} qr  \mathop{\sum \sum }_{d \, d^\prime \mid q }  (d, d^\prime ) \mathop{\sum \sum}_{\substack { m, m^\prime \sim M_1 \\ (d, d^\prime ) \mid m-m^\prime}} 1 \\
	& \ll \frac{N_0|\mathfrak{J}|}{n_1^2 M_1^{1/2}  }  \sum_{q_2 \sim C/q_1} qr \mathop{\sum \sum }_{d \, d^\prime \mid q } \left( M_1 (d, d^\prime ) + \frac{M_1^{1/2}C}{X^{1/2}}\right) \\
	& \ll\frac{N_0|\mathfrak{J}|C^2r}{n_1^2 M_1^{1/2}q_1}\left( M_1 C + \frac{M_1^{1/2}C}{X^{1/2}}\right).
	 \end{align*} 
	  Thus we get the first part of the lemma. Substituting the above bound of $\Omega_0$ in place of $\Omega$ in \eqref{R(X) after cauchy} we get 
	  \begin{align*}
	  R^{0}(X)\ll  \mathop{\sup_{ \substack{M_1\ll M_0 \\ C \ll Q}}} \, \frac{X^{5/12}|\mathfrak{J}|^{1/2} }{r^{2/3}C^{3/2}} \sum_{\frac{n_1}{(n_1,r)} \ll C} n_1^{1/3}\Theta^{1/2}\sum_{\frac{n_1}{(n_1,r)}|q_1|(n_1r)^\infty}\frac{\sqrt{N_0r}C^{3/2}M_1^{1/4}}{n_1q_1^{1/2}},
	  \end{align*}
	  Using the trivial bound for $q_1$ and replacing the range for $n_1$ by the longer range $n_1 \ll Cr$, we arrive at 
	  \begin{align*}
	  	\mathop{\sup_{ \substack{M_1\ll M_0 \\ C \ll Q}}}{r^{-1/6}X^{5/12}N_0^{1/2}M_1^{1/4}|\mathfrak{J}|^{1/2} }\sum_{n_1\ll Cr}\frac{(n_1,r)^{1/2}}{n_1^{7/6}}\Theta^{1/2}.
	  \end{align*}
	  Using \eqref{theta bound} to bound $n_1$-sum, we see that
	  \begin{align*}
	  	 R^{0}(X)\ll \mathop{\sup_{ \substack{M_1\ll M_0 \\ C \ll Q}}}{r^{-1/6}X^{5/12}N_0^{2/3}M_1^{1/4}|\mathfrak{J}|^{1/2}}.
	  \end{align*}
	  Using $N_0 \ll K^{3/2}X^{1/2}r$ and $M_1 \ll M_0= X^\epsilon  \ \max \left(\alpha^2\beta^2 C^2 X^{2\beta-1}, \  K \right)$, we arrive at
	  \begin{align}\label{zero freq bound }
	  	 R^{0}(X)\ll \mathop{\sup_{ \substack{ C \ll Q}}}{X^{3/4}r^{1/2}KM_0^{1/4}|\mathfrak{J}|^{1/2}}\ll r^{1/2}(\alpha\beta)^{1/2}X^{3/4+3\beta/4-3\eta/4}.
	  \end{align}
	Hence the lemma follows.
 \end{proof}
\subsection{Non-zero frequency} In this subsection, we will estimate $R(X)$ for  $n_2 \neq 0$. Let $\Omega_{\neq 0}$ denotes the contribution of $n_2 \neq 0$ to $\Omega$ in \eqref{omega 1}. Let $R^{\neq 0}(X)$ denotes the part of $R(X)$ in \eqref{R(X) after cauchy} corresponding to $\Omega_{\neq 0}$. We have the following lemma.
\begin{lemma}\label{omega nonzero bound}
	We have 
	$$\Omega_{\neq 0} \ll\frac{ (XK)^{1/2}r^2\vert \mathfrak{J}\vert C^3} {n_{1}^2q_1M_1^{1/2}} \left(\frac{C^2n_1}{q_1^2}+\frac{Cn_1M_1}{q_1}+M_1^2 \right). $$
\end{lemma}
\begin{proof}
We start by analyzing $\Omega$ in \eqref{omega 2}. 	Using Lemma \ref{bound for c} in place of $\mathfrak{C}$ in \eqref{omega 2}, we get
	\begin{align*}
	\Omega^{\neq 0} & \ll \frac{q_{1}^2 N_0r} {n_{1}^3M_1^{1/2}} \mathop{\sum \sum}_{\substack{d_{2} \mid q_{2} \\ d_{2}^{\prime} \mid q_{2}^{\prime}}} d_{2} d_{2}^{\prime} \, \mathop{\sum \sum }_{\substack{q_2, q_{2}^{\prime} \sim \frac{C}{q_{1}}  } }  \mathop{ \mathop{\sum \ \sum \ \  \ \sum}_{m,m^{\prime} \sim M_{1} \ n_2 \in \mathbf{Z}-\{0\}}}_{\substack{ n_{1} q_{2}^{\prime} d_{2}^{\prime}- m n_{2} \, \equiv \,  0 \, \rm mod \, d_{2} \\  n_{1} q_{2} d_{2}+ m^{\prime} n_{2} \, \equiv \,  0 \, \rm mod \, d_{2}^{\prime}}} (m,n_{1})\vert \mathfrak{J}\vert .
	\end{align*}
	Further writing $q_2d_2$ in place of $q_2$ and $q_2^{\prime}d_2^{\prime}$ in place of $q_2^{\prime}$, we arrive at
	\begin{align}\label{omega nonzero}
	\Omega^{\neq 0} & \ll \frac{q_{1}^2 N_0r} {n_{1}^3M_1^{1/2} } \mathop{\sum \sum}_{d_{2}, d_{2}^{\prime} \ll C/q_{1} } d_{2} d_{2}^{\prime} \, \mathop{\sum \sum }_{\substack{q_{2} \sim \frac{C}{d_{2}q_{1}} \\ q_{2}^{\prime} \sim \frac{C}{d_{2}^{\prime} q_{1}}}}  \mathop{ \mathop{\sum \ \sum \ \  \ \sum}_{m,m^{\prime} \sim M_{1} \ n_2 \in \mathbf{Z}-\{0\}}}_{\substack{ n_{1} q_{2}^{\prime} d_{2}^{\prime}- m n_{2} \, \equiv \,  0 \, \rm mod \, d_{2} \\  n_{1} q_{2} d_{2}+ m^{\prime} n_{2} \, \equiv \,  0 \, \rm mod \, d_{2}^{\prime}}} (m,n_{1})\vert \mathfrak{J}\vert . 
	\end{align}
	Next, we  count the number of $m$ and $ m^{\prime}$ in the above expression.  We have   
	\begin{align*}
	\sum_{\substack{m \sim M_{1} \\ n_{1} q_{2}^{\prime} d_{2}^{\prime}- m n_{2} \, \equiv \,  0 \, \rm mod \, d_{2}}} (m,n_{1}) & = \sum_{\ell \mid n_{1}} \ell \,  \sum_{\substack{m \sim M_{1}/\ell \\ n_{1} q_{2}^{\prime} d_{2}^{\prime} \bar{\ell}- m n_{2} \, \equiv \,  0 \, \rm mod \, d_{2}}} 1  
	& \ll (d_{2},n_{2}) \, \left(n_{1}+\frac{M_{1}}{d_{2}}\right) \notag.
	\end{align*} 
	In the above estimate we have used the fact $(d_2,n_2)=1$. Counting the number of $m$ in a similar fashion we get that  $m$-sum  and $m^{\prime}$-sum in \eqref{omega nonzero}  is dominated by 
	$$ X^{\epsilon}(d_{2}^{\prime}, n_{1} q_{2} d_{2}) \, (d_{2}, n_{2}) \left(n_{1}+\frac{M_{1}}{d_{2}}\right) \left(1+\frac{M_{1}}{d_{2}^{\prime}}\right).$$ 
	Now substituting the above bound in \eqref{omega nonzero}, we arrive at
	\begin{align*}
	\frac{q_{1}^2 N_0r\vert \mathfrak{J}\vert} {n_{1}^3M_1^{1/2}} \mathop{\sum \sum}_{d_{2}, d_{2}^{\prime} \ll \frac{C}{q_1}} d_{2} d_{2}^{\prime} \, \mathop{\sum \sum }_{\substack{q_{2} \sim \frac{C}{d_{2}q_{1}} \\ q_{2}^{\prime} \sim \frac{C}{d_{2}^{\prime} q_{1}}}} \sum_{1<|n_2|\ll \tilde{N}}(d_{2}^{\prime}, n_{1} q_{2} d_{2})  (d_{2}, n_{2}) \left(n_{1}+\frac{M_{1}}{d_{2}}\right) \left(1+\frac{M_{1}}{d_{2}^{\prime}}\right).
	\end{align*}
	Now summing  over $n_2$, $q_2^{\prime}$, we get the following expression:
	\begin{align*}
	\Omega^{\neq 0} \ll \frac{q_{1} N_0r\tilde{N}\vert \mathfrak{J}\vert C} {n_{1}^3M_1^{1/2} } \mathop{\sum \sum}_{d_{2}, d_{2}^{\prime} \ll C/q_{1} } d_{2}  \, \mathop{ \sum }_{\substack{q_{2} \sim \frac{C}{d_{2}q_{1}}  }} (d_{2}^{\prime}, n_{1} q_{2} d_{2}) \,  \left(n_{1}+\frac{M_{1}}{d_{2}}\right) \left(1+\frac{M_{1}}{d_{2}^{\prime}}\right).
	\end{align*}
	Next we  sum over $d_2^{\prime}$ to  get
	\begin{align*}
	\Omega^{\neq 0} & \ll \frac{q_{1} N_0r\tilde{N}\vert \mathfrak{J}\vert C} {n_{1}^3M_1^{1/2}} \mathop{\sum }_{d_{2} \ll C/q_{1} } d_{2}  \, \mathop{ \sum }_{\substack{q_{2} \sim \frac{C}{d_{2}q_{1}}  }}    \left(n_{1}+\frac{M_{1}}{d_{2}}\right) \left(\frac{C}{q_1}+M_1\right).
	\end{align*}
	Finally executing the remaining sums,  we get  
	\begin{align}\label{omega nonzero final }
	\Omega^{\neq 0}\ll \frac{ N_0r\tilde{N}\vert \mathfrak{J}\vert C^2} {n_{1}^3M_1^{1/2}} \left(\frac{Cn_1}{q_1}+M_1\right) \left(\frac{C}{q_1}+M_1\right).
	\end{align}
	Lastly using $N_0\tilde{N}=\frac{n_1r}{q_1}C\sqrt{XK}$, and expanding the brackets gives us the lemma.
\end{proof}
Now we will estimate $R^{\neq 0}(X)$. We will analyze it in two cases.

\subsubsection{\bf Estimates for small  $q \sim C$} Let $R_{small}^{\neq 0}(X)$ denotes the contribution of $C \ll  X^{1/2-{\beta/2}}$ and $n_2 \neq 0$ to $R(X)$ in \eqref{R(X) after cauchy}. Then
  \begin{lemma} \label{R(X) bound for small}
We have 
\begin{align*}
 R_{small}^{\neq 0}(X) &\ll r(\alpha \beta)^{3/2}X^{3/4+3\beta/4-\eta/2}.
\end{align*}
\end{lemma}
\begin{proof}
	Using Lemma \ref{omega nonzero bound} for $\Omega$ in the expression of $R(X)$ in \eqref{R(X) after cauchy}, we see that
	\begin{align*}
 R_{small}^{\neq 0}(X) \ll &\mathop{\sup_{ \substack{M_1\ll M_0 \\ C \ll Q}}}\frac{X^{5/12} }{C^{3/2}r^{2/3}}\sum_{\frac{n_1}{(n_1,r)}\ll C}n_1^{1/3}\Theta^{1/2} \\ 
 & \times \sum_{\frac{n_1}{(n_1,r)}|q_1|(n_1r)^\infty} \frac{ r\vert \mathfrak{J}\vert^{1/2} C^{3/2}(XK)^{1/4}} {n_{1}q_1^{1/2}M_1^{1/4}}\left(\frac{Cn_1^{1/2}}{q_1}+\frac{(Cn_1M_1)^{1/2}}{q_1^{1/2}}+M_1 \right).
	\end{align*}  
	Second and third terms of the right hand side can be estimated easily. Infact, let's consider the expression corresponding to the third term ($M_1$)
	\begin{align*}
	 \mathop{\sup_{ \substack{M_1\ll M_0 \\ C \ll Q}}}\frac{X^{5/12} }{C^{3/2}r^{2/3}}\sum_{\frac{n_1}{(n_1,r)}\ll C}n_1^{1/3}\Theta^{1/2} 
	 \sum_{\frac{n_1}{(n_1,r)}|q_1|(n_1r)^\infty}\frac{ r\vert \mathfrak{J}\vert^{1/2} C^{3/2}(XK)^{1/4}} {n_{1}q_1^{1/2}M_1^{1/4}}M_1 
	\end{align*}
	 Executing the $q_1$ sum trivially, we arrive at
	\begin{align*}
	\mathop{\sup_{ \substack{M_1\ll M_0 \\ C \ll Q}}}\frac{X^{5/12} r\vert \mathfrak{J}\vert^{1/2}}{C^{3/2}r^{2/3}}\sum_{\frac{n_1}{(n_1,r)}\ll C}\frac{(n_1,r)^{1/2}}{n_1^{7/6}} \Theta^{1/2} 
	{ (XK)^{1/4}C^{3/2} } M_1^{3/4}
	\end{align*}
	Using \eqref{theta bound} for $n_1$-sum, we get
	
		\begin{align}\label{third term}
	\mathop{\sup_{ \substack{M_1\ll M_0 \\ C \ll Q}}}\frac{X^{5/12} r\vert \mathfrak{J}\vert^{1/2}}{C^{3/2}r^{2/3}}N_0^{1/6} 
	{ (XK)^{1/4}C^{1/2} C} M_1^{3/4}
	\end{align}
	Now using $M_1 \ll M_0 \ll (\alpha \beta)^2X^{(1-\beta)}X^{2 \beta -1}$, as $C \ll X^{1/2- \beta/2}$,  $N_0 \ll X^{1/2}K^{3/2}r$ and Lemma \eqref{second L^2 bound} for $\mathfrak{J}$, we get
\begin{align}
	r^{1/2}(\alpha \beta)^{3/2}X^{3/4+3\beta/4-\eta/2}.
\end{align}
Now consider the expression corresponding to the second term ${(Cn_1M_1)^{1/2}}/{q_1^{1/2}}$
	\begin{align*}
\mathop{\sup_{ \substack{M_1\ll M_0 \\ C \ll Q}}}\frac{X^{5/12} }{C^{3/2}r^{2/3}}\sum_{\frac{n_1}{(n_1,r)}\ll C}n_1^{1/3}\Theta^{1/2} 
\sum_{\frac{n_1}{(n_1,r)}|q_1|(n_1r)^\infty}\frac{ r\vert \mathfrak{J}\vert^{1/2} C^{3/2}(XK)^{1/4}} {n_{1}q_1^{1/2}M_1^{1/4}}\frac{(Cn_1M_1)^{1/2}}{q_1^{1/2}}
\end{align*}
 Executing the $q_1$ sum trivially, we arrive at
\begin{align}\label{second term}
\mathop{\sup_{ \substack{M_1\ll M_0 \\ C \ll Q}}}\frac{X^{5/12} r\vert \mathfrak{J}\vert^{1/2}}{C^{3/2}r^{2/3}}\sum_{\frac{n_1}{(n_1,r)}\ll C}\frac{(n_1,r)}{n_1^{7/6}} \Theta^{1/2} 
{ (XK)^{1/4}C^2} M_1^{1/4}
\end{align}
Note that
\begin{align*}
\sum_{n_1\ll Cr}\frac{(n_1,r)}{n_1^{7/6}}\Theta^{1/2} \ll \left[\sum_{n_1 \ll Cr}\frac{(n_1,r)^2}{n_1}\right]^{1/2}\left[\mathop{\sum \sum}_{n_{1}^{2} n_{2} \leq N_0} \frac{\vert \lambda_{\pi}(n_{1},n_{2})\vert ^{2} }{(n_1^2n_2)^{2/3}}\right]^{1/2}\ll \,r^{1/2}N_0^{1/6}.
\end{align*}
Hence using the above bound,  $M_1 \ll M_0 \ll (\alpha \beta)^2X^{(1-\beta)}X^{2 \beta -1}$, $C \ll X^{1/2- \beta/2}$, $N_0 \ll X^{1/2}K^{3/2}r$ and Lemma \eqref{second L^2 bound}, we arrive at
\begin{align}
r(\alpha \beta)^{1/2}X^{1-\eta/2}.
\end{align}
Next we consider the expression corresponding to the first term ${Cn_1^{1/2}}/{q_1}$ 
	\begin{align}\label{first}
\mathop{\sup_{ \substack{M_1\ll M_0 \\ C \ll Q}}}\frac{X^{5/12} }{C^{3/2}r^{2/3}}\sum_{\frac{n_1}{(n_1,r)}\ll C}n_1^{1/3}\Theta^{1/2} 
\sum_{\frac{n_1}{(n_1,r)}|q_1|(n_1r)^\infty}\frac{ r\vert \mathfrak{J}\vert^{1/2} C^{3/2}(XK)^{1/4}} {n_{1}q_1^{1/2}M_1^{1/4}}\frac{Cn_1^{1/2}}{q_1}
\end{align}
 Estimating the above expression like the second term, we arrive at
\begin{align}\label{first term bound 1}
\mathop{\sup_{ \substack{M_1\ll M_0 \\ C \ll Q}}}\frac{X^{5/12} r\vert \mathfrak{J}\vert^{1/2}}{C^{3/2}r^{2/3}}r^{1/2}N_0^{1/6}{(XK)^{1/4}C^2} M_1^{1/4} \times \frac{\sqrt{C}}{\sqrt{M_1}}
\end{align}
Now if $C$ is not of the size $X^{1/2-\beta/2-\eta/2}$, then by the arguments in Lemma \ref{second L^2 bound}, we may assume that $M_1 \asymp M_0$. Hence we get 
\begin{align}
&\mathop{\sup_{ \substack{C \ll X^{1/2-\beta/2}}}}\frac{X^{5/12} r\vert \mathfrak{J}\vert^{1/2}}{C^{3/2}r^{2/3}}r^{1/2}N_0^{1/6}{(XK)^{1/4}C^2} M_0^{1/4} \times \frac{\sqrt{C}}{\sqrt{M_0}} \notag \\
&\ll \mathop{\sup_{ \substack{C \ll X^{1/2-\beta/2}}}}{X^{5/12-\beta/2} r^{1/3}}r^{1/2}N_0^{1/6}{(XK)^{1/4}} \frac{{C}}{{M_0^{1/4}}} \notag \\
&\ll {X^{5/12-\beta/2} r}X^{1/12}K^{1/4}{(XK)^{1/4}} \frac{{X^{1/4-\beta/4}}}{{X^{\beta/2-1/4}}} \ll rX^{5/4-3\beta/4-\eta/2}.
\end{align}
Next, we are left with the case when $C \asymp X^{1/2-\beta/2-\eta/2}$. If $M_1 \gg C/q_1$,  then, using similar calculations, we get back to the previous case. Now we assume that $M_1\ll C/q_1$.  
In this case, $$\frac{3 \left(X N_{0}\right)^{1/3}}{Cr^{1/3}} \asymp X^{\beta},$$
and hence $N_0 \asymp C^3X^{3\beta}r/X$, and consequently 
$$\tilde{N} \asymp  \frac{\sqrt{XK} Cn_1r}{N_{0}q_1}=\frac{CX^{1/3}r^{2/3}n_1}{q_1N_0^{2/3}} \asymp \frac{Xn_1}{CX^{2\beta}q_1}.$$
In this case, we adopt a different strategy for counting. We first consider the congruence relation in \eqref{omega nonzero}
$$n_{1} q_{2}^{\prime} d_{2}^{\prime}- m n_{2} \, \equiv \,  0 \, \rm mod \, d_{2}.$$
Note that 
$$n_{1} q_{2}^{\prime} d_{2}^{\prime}- m n_{2} \ll Cn_1/q_1+M_1\tilde{N} \ll Cn_1/q_1+Cn_1/q_1.$$
Let $d_{2}^{\prime} \sim D^{\prime}\gg D \sim d_2$ and 
$$n_{1} q_{2}^{\prime} d_{2}^{\prime}-mn_{2} = h d_{2}, \ \ \ \text{with} \ \ h \ll Cn_1/(q_1D).$$
Thus \eqref{omega nonzero} can be rewritten as 
	\begin{align*}
\Omega^{\neq 0} & \ll \frac{q_{1}^2 N_0r} {n_{1}^3M_1^{1/2} } \mathop{\sum \sum}_{d_{2}, d_{2}^{\prime} \ll C/q_{1} } d_{2} d_{2}^{\prime} \, \mathop{\sum \sum }_{\substack{q_{2} \sim \frac{C}{d_{2}q_{1}} \\ h \ll Cn_1/q_1D}}  \mathop{ \mathop{\sum \ \sum \ \  \ \sum}_{m,m^{\prime} \sim M_{1} \ n_2 \in \mathbf{Z}-\{0\}}}_{\substack{h d_{2}+ m n_{2} \, \equiv \,  0 \, \rm mod \, d_{2}^{\prime} \\  n_{1} q_{2} d_{2}+ m^{\prime} n_{2} \, \equiv \,  0 \, \rm mod \, d_{2}^{\prime}}} (m,n_{1})\vert \mathfrak{J}\vert . 
\end{align*}
 Using the second congruence equation, the number of $m^{\prime}$  comes out to be $$O\left((n_{2},d_{2}^{\prime}) (1+{M_{1}}/{D^{\prime}})\right).$$
 The first congruence equation provides either the count for $d_{2}$ which comes to be $O((d_{2}^{\prime},h) D/D^{\prime})$ for $h\neq 0$ or  count  $d_{2}^{\prime}$ which comes out to be $O(X^{\epsilon})$ for $h =0$. Thus we arrive at 
 \begin{align*}
 	 \frac{q_{1}^2 N_0r} {n_{1}^3M_1^{1/2}X^{\beta} }\sum_{d_2^{\prime} \sim D^{\prime}}D^2 \mathop{\sum \sum }_{\substack{q_{2} \sim \frac{C}{d_{2}q_{1}} \\ h \ll Cn_1/q_1D}} \sum_{m \sim M_1}\sum_{0 <n_2 \ll \tilde{N}}(n_{2},d_{2}^{\prime}) (h, d_{2}^{\prime}) (m,n_{1}) \left(1+\frac{M_{1}}{D^{\prime}}\right)
 \end{align*}
 First summing over $n_2$, and then over $m$ and $d_2^{\prime}$ we arrive at 
 \begin{align*}
 	 \frac{q_{1}^2 N_0r} {n_{1}^3M_1^{1/2}X^{\beta} }M_1\tilde{N}D^{\prime}D^2\mathop{\sum \sum }_{\substack{q_{2} \sim \frac{C}{d_{2}q_{1}} \\ h \ll Cn_1/q_1D}}\left(1+\frac{M_{1}}{D^{\prime}}\right)
 \end{align*}
 Next summing over $q_2$ and $h$ we get 
  \begin{align*}
 \frac{ N_0r} {n_{1}^2X^{\beta}} M_1^{1/2}\tilde{N}C^2(D^{\prime}+M_1)\ll \frac{ N_0r} {n_{1}^2q_1X^{\beta}} M_1^{1/2}\tilde{N}C^3\ll \frac{(XK)^{1/2}C^4M_1^{1/2}r^2}{n_1q_1^2}
 \end{align*}
 We now substitute the above bound in place of $\Omega$ in \eqref{R(X) after cauchy} to get
 \begin{align}
 \mathop{\sup_{ \substack{M_1\ll M_0 \\ C \ll Q}}}\frac{X^{5/12} }{C^{3/2}r^{2/3}}\sum_{\frac{n_1}{(n_1,r)}\ll C}n_1^{1/3}\Theta^{1/2}\sum_{\frac{n_1}{(n_1,r)}|q_1|(n_1r)^\infty}\frac{(XK)^{1/4}C^2M_1^{1/4}r}{n_1^{1/2}q_1} ,
 \end{align} 
 Now estimating it as the second term case we get,
 $$r(\alpha \beta)^{1/2}X^{1-\eta/2}.$$
  
Combining all the cases, we get the lemma.  
  \end{proof}

\subsubsection{\bf Estimates for generic  $q \sim C$} Let $R_{generic}^{\neq 0}(X)$ denotes the contribution of $C \gg  X^{{1/2}-{\beta/2}+ \epsilon}$ and $n_2 \neq 0$ to $R(X)$ in \eqref{R(X) after cauchy}. Then
\begin{lemma}\label{R(X) generic bound}
	We have
	\begin{align*}
		R_{generic}^{\neq 0}(X) \ll  r(\alpha\beta)^{3/2}X^{3/4+7\beta/12+5\eta/12}.
	\end{align*}
	\end{lemma}
\begin{proof}
 Proof of this lemma goes along the same line as that of Lemma \ref{R(X) bound for small}. In the proof, we will exploit the fact that for large $C$, we have a better bound for $\mathfrak{J}$. Plugging in Lemma \ref{second L^2 bound} in the proof of Lemma \ref{omega nonzero bound}, we see that 
 \begin{align*}
 	\Omega^{\neq 0}\ll \frac{r^2(XK)^{1/2}C^{3}} {n_{1}^2M_1^{1/2}q_1X^{\beta}}\frac{C^{1/3}}{(XK)^{1/6}} \left(\frac{C^2n_1}{q_1^2}+\frac{Cn_1M_1}{q_1}+M_1^2 \right).
 \end{align*}
 Hence
 \begin{align*}
& R_{generic}^{\neq 0}(X) \ll \mathop{\sup_{ \substack{M_1\ll M_0 \\ C \ll Q}}}\frac{X^{5/12} }{C^{3/2}r^{2/3}}\sum_{\frac{n_1}{(n_1,r)}\ll C}n_1^{1/3}\Theta^{1/2} \\ 
 & \times \sum_{\frac{n_1}{(n_1,r)}|q_1|(n_1r)^\infty}\frac{rX^{1/4}K^{1/4}C^{3/2}} {n_{1}M_1^{1/4}q_1^{1/2}X^{\beta/2}}\frac{C^{1/6}}{(XK)^{1/12}} \left(\frac{Cn_1^{1/2}}{q_1}+\frac{(Cn_1M_1)^{1/2}}{q_1^{1/2}}+M_1 \right),
 \end{align*} 
 We can treat second and third term as in Lemma \eqref{R(X) bound for small}. Infact, by \eqref{third term}, the  expression corresponding to the third term $M_1$ is dominated by
 
 \begin{align*}
 \mathop{\sup_{ \substack{M_1\ll M_0 \\ C \ll Q}}}&\frac{X^{5/12} r}{C^{3/2}r^{2/3}}\frac{N_0^{1/6} 
 	{ (XK)^{1/4}C^{3/2}} M_1^{3/4}}{X^{\beta/2}}\frac{C^{1/6}}{(XK)^{1/12}} \\
& \ll \frac{r^{1/3}X^{5/12}N_0^{1/6} 
 	{ (XK)^{1/6}Q^{1/6}} M_0^{3/4}}{X^{\beta/2}}.
 \end{align*}
 Now using  $N_0 \ll X^{1/2}K^{3/2}r$ and $ M_0 \ll (\alpha\beta)^{2}Q^2X^{2\beta-1}$, we arrive at 
  \begin{align}\label{R(X) for third term}
 	r^{1/2}(\alpha\beta)^{3/2}X^{1/2-\beta/2}X^{3\beta/2-7/12}K^{5/12}Q^{5/3} \ll r^{1/2}(\alpha\beta)^{3/2}X^{3/4+7\beta/12+5\beta/12}.
 \end{align}
Similarly, using \eqref{second term}, the  expression corresponding to the second term  is dominated by
 \begin{align}\label{R(X) for secondd term}
 	 \mathop{\sup_{ \substack{ C \ll Q}}}\frac{X^{5/12} r}{C^{3/2}r^{2/3}}\frac{r^{1/2}N_0^{1/6} 
 		{ (XK)^{1/4}C^{3/2}} M_0^{3/4}}{X^{\beta/2}}\frac{C^{1/6}}{(XK)^{1/12}}\frac{C^{1/2}}{M_0^{1/2}}.
 \end{align}
 Replacing $C$ by $Q$, and accordingly $M_0$ by $(\alpha\beta)^{2}Q^2X^{2\beta-1}$, we see that 
 $$\frac{Q^{1/2}}{(\alpha\beta)QX^{\beta-1/2}} \ll 1,$$
 provided $3\beta+\eta >1$.
 And hence we get 
 \begin{align}\label{R(X) for second term}
r(\alpha\beta)^{3/2}X^{3/4+7\beta/12+5\eta/12}.
 \end{align}
 Finally, considering the expression corresponding to the first term ${Cn_1^{1/2}}/{q_1}$, we have 
 \begin{align}
 \mathop{\sup_{ \substack{M_1\ll M_0 \\ C \ll Q}}}\frac{X^{5/12} }{C^{3/2}r^{2/3}}\sum_{\frac{n_1}{(n_1,r)}\ll C}n_1^{1/3}\Theta^{1/2} 
 \sum_{\frac{n_1}{(n_1,r)}|q_1|(n_1r)^\infty}\frac{rX^{1/4}K^{1/4}C^{3/2}} {n_{1}M_1^{1/4}q_1^{1/2}X^{\beta/2}}\frac{C^{1/6}}{(XK)^{1/12}}\frac{Cn_1^{1/2}}{q_1}
 \end{align}
 Bounding $\frac{Cn_1^{1/2}}{q_1}$ by $\frac{Cn_1^{1/2}}{q_1^{1/2}}$ and estimating the resulting expression like the second term, we arrive at
  \begin{align*}
 \mathop{\sup_{ \substack{ C \ll Q}}}\frac{X^{5/12} r}{C^{3/2}r^{2/3}}\frac{r^{1/2}N_0^{1/6} 
 	{ (XK)^{1/4}C^{3/2}} M_0^{3/4}}{X^{\beta/2}}\frac{C^{1/6}}{(XK)^{1/12}}\frac{C^{1/2}}{M_0^{1/2}}\frac{C^{1/2}}{M_0^{1/2}}.
 \end{align*}
 This is similar to \eqref{R(X) for secondd term}, hence we get the same bound as in \eqref{R(X) for second term}.
 Finally combining all the cases, we get the lemma. 
\end{proof}
\subsection{Conclusion} We now pull together the bounds from Lemma \ref{R(X) zero frq bound}, Lemma \ref{R(X) bound for small} and Lemma \eqref{R(X) generic bound} to get 
\begin{align*}
	R(X) \ll r(\alpha\beta)^{3/2}X^{3/4+3\beta/4-3\eta/4}+r(\alpha \beta)^{3/2}X^{3/4+3\beta/4-\eta/2}+ r(\alpha\beta)^{3/2}X^{3/4+7\beta/12+5\eta/12}.
\end{align*}
 Here second term dominates the first one. Hence we get
\begin{align*}
R(X) \ll r(\alpha \beta)^{3/2}X^{3/4+3\beta/4-\eta/2}+ r(\alpha\beta)^{3/2}X^{3/4+7\beta/12+5\eta/12}.
\end{align*}
Equating the above exponents, we get the optimal choice for $\eta$
$$\eta=2\beta/11.$$
And hence 
$$R(X) \ll  (\alpha \beta)^{3/2} rX^{3/4+29\beta/44}.$$
Thus we have proved Theorem \ref{gl3thm21}.
\end{proof}


\section*{Acknowledgements}
Authors are thankful to Prof. Ritabrata Munshi for sharing his ideas, explaining his methods and his support throughout the work.  Authors also wish to thank Prof. Satadal Ganguly for his encouragement and  constant support. Authors are grateful to  Stat-Math Unit, Indian Statistical Institute, Kolkata for providing wonderful research environment. During this work, S. Singh was supported by D.S.T. inspire faculty fellowship no. DST/INSPIRE/$04/2018/000945$.

\end{document}